\documentclass[onefignum,onetabnum]{siamart220329}

\usepackage{enumitem}
\setlist[enumerate]{label=(\roman*)}
\usepackage{amsfonts}
\usepackage{amssymb}
\usepackage{mathrsfs}
\usepackage{dsfont}
\usepackage{mathtools}
\usepackage{algorithm}
\usepackage{algorithmic}
\usepackage{epstopdf}
\usepackage{subfigure} 

\newcommand\dd{\mathop{}\!\mathrm{d}}
\newcommand\N{\mathbb{N}}

\newcommand\R{\mathbb{R}}

\newcommand\DD{\mathcal{D}}
\newcommand\GG{\mathcal{G}}
\newcommand\HH{\mathcal{H}}

\newcommand\KK{\mathcal{K}}
\newcommand\LL{\mathscr{L}}

\newcommand\QQ{\mathcal{Q}}
\newcommand\RR{\mathcal{R}}
\newcommand\TT{\mathcal{T}}
\newcommand\XX{\mathcal{X}}

\newcommand\Id{\mathrm{Id}}
\newcommand\cl{\mathrm{cl}}
\renewcommand\ker{\mathrm{ker}}
\newcommand\range{\mathrm{range}}
\newcommand\supp{\mathrm{supp}}

\newcommand\SSS{\mathcal{S}}

\newcommand\dualspace{^\star}
\newcommand{\weakly}{\rightharpoonup}

\newcommand{\D}{D}

\makeatletter
\newcommand{\oset}[3][0ex]{%
  \mathrel{\mathop{#3}\limits^{
    \vbox to#1{\kern-2\ex@
    \hbox{$\scriptstyle#2$}\vss}}}}
\makeatother

\newcommand{\compactly}{\oset{\textup{c}}{\hookrightarrow}}

\newcommand{\WOT}{\smash{\overset{\normalfont\textsc{wot}}{\to}}}

\DeclareMathOperator{\dist}{dist}

\DeclarePairedDelimiter\parens()

\DeclarePairedDelimiter\bracks[]
\DeclarePairedDelimiter\abs{\lvert}{\rvert}
\DeclarePairedDelimiter\norm{\lVert}{\rVert}
\DeclarePairedDelimiterX\innerp[2](){#1,#2}
\DeclarePairedDelimiterX\dual[2]{\langle}{\rangle}{#1,#2}
\providecommand\given{\nonscript\;\delimsize|\nonscript\;\mathopen{}}
\DeclarePairedDelimiterX\set[1]\{\}{#1}

\newsiamthm{assumption}{Assumption}
\crefname{assumption}{Assumption}{Assumptions}
\newsiamremark{remark}{Remark}
\newsiamthm{Algorithm}{Algorithm}
\crefname{ALC@unique}{Step}{Steps}

\usepackage[normalem]{ulem}

\definecolor{darkgreen}{rgb}{0,0.5,0}
\definecolor{darkred}{rgb}{0.8,0,0}

\usepackage{marginnote}
\usepackage[normalem]{ulem}
\newcommand{\stkout}[1]{\ifmmode\text{\sout{\ensuremath{#1}}}\else\sout{#1}\fi}

\usepackage{etoolbox}
\makeatletter
\patchcmd{\@addmarginpar}{\ifodd\c@page}{\ifodd\c@page\@tempcnta\m@ne}{}{}
\makeatother
\reversemarginpar

\DeclareRobustCommand{\[}{\begin{equation*}}
\DeclareRobustCommand{\]}{\end{equation*}}

\headers{Energy Space Newton Differentiability for Obstacle Problems}
{Constantin Christof and Gerd Wachsmuth}

\title{Energy Space Newton Differentiability for Solution Maps of Unilateral and Bilateral Obstacle Problems%
}

\author{Constantin Christof\thanks{%
Technische Universit\"at M\"unchen,
Faculty of Mathematics, M17,
85748 Garching bei M\"unchen, Germany,
\url{https://www-m17.ma.tum.de/Lehrstuhl/ConstantinChristof},
\email{christof@ma.tum.de}%
}%
\and
Gerd Wachsmuth\thanks{%
Brandenburgische Technische Universit\"at Cottbus--Senftenberg, 
Institute of Mathematics, 
03046 Cott\-bus, Germany, 
\url{https://www.b-tu.de/fg-optimale-steuerung},
\email{gerd.wachsmuth@b-tu.de},
ORCID: \href{https://orcid.org/0000-0002-3098-1503}{0000-0002-3098-1503}%
}%
}

\begin{document}

\maketitle

\begin{abstract}
We prove that the solution operator of the classical unilateral obstacle problem 
on a nonempty open bounded set $\Omega \subset \R^d$, $d \in \N$,
is
Newton differentiable as a function from $L^p(\Omega)$ to $H_0^1(\Omega)$ whenever
$\max(1, 2d/(d+2)) < p \leq \infty$.
By exploiting this Newton differentiability property, 
results on angled subspaces in $H^{-1}(\Omega)$,
and a formula for orthogonal projections onto direct sums,
we further show that the solution map of the classical bilateral obstacle problem 
is Newton differentiable as a function from $L^p(\Omega)$ to $H_0^1(\Omega)\cap L^q(\Omega)$
whenever 
$\max(1, d/2) < p \leq \infty$ and $1 \leq q <\infty$. For both the unilateral and the bilateral case, we provide explicit formulas for the 
Newton derivative. As a concrete application example for our results, 
we consider the numerical solution of an optimal control problem with 
 $H_0^1(\Omega)$-controls and box-constraints by means of a semismooth Newton method. 
\end{abstract}




\section{Introduction and summary of results}
\label{sec:1}
This paper is concerned with Newton differentiability properties of 
solution maps
$S \colon H^{-1}(\Omega) \to H_0^1(\Omega)$, \mbox{$u \mapsto y$,}
of obstacle-type variational inequalities (VIs) of the form 
\begin{equation*}
\label{eq:VI}
\tag{\textup{VI}}
y \in K,\qquad 
\dual*{-\Delta y   - u}{ v - y }_{H_0^1(\Omega)} \geq 0\qquad \forall v \in K.
\end{equation*}
Here, $\Omega \subset \R^d$, $d \in \mathbb{N}$, is a
nonempty open bounded set; $H_0^1(\Omega)$, $H^{-1}(\Omega)$, 
$\langle \cdot,\cdot\rangle_{H_0^1(\Omega)}$, and
$-\Delta$ are defined as usual
(see \cref{sec:2}); and $K$ is a nonempty set of the 
type
\begin{align}
\label{def:K}
K := \{v \in H_0^1(\Omega) \mid \psi \leq v \leq \phi \text{ a.e.\ in }\Omega\}
\end{align}
involving measurable functions $\psi, \phi\colon \Omega \to [-\infty, \infty]$.
The main contributions of this work, that can be seen as a continuation of 
our previous study \cite{ChristofWachsmuth2023}, are as follows:
First, we prove that the 
Newton differentiability result in \cite[Corollary~3.9]{ChristofWachsmuth2023},
which yields that the solution map $S\colon u \mapsto y$
of \eqref{eq:VI} 
is Newton differentiable as a function from 
$L^p(\Omega)$ to $L^q(\Omega)$
for certain $p$ and $q$ when $\phi \equiv \infty$ holds, 
can be improved
to obtain that 
$S$ is even Newton differentiable 
as a function from $L^p(\Omega)$ to $H_0^1(\Omega)$
whenever $\max(1, 2d/(d+2)) < p \leq \infty$ and $\phi \equiv \infty$.
Second, we show that this energy space Newton differentiability result for the unilateral 
case 
makes it possible to prove that the solution map of the bilateral obstacle problem, i.e.,
the problem \eqref{eq:VI} with $\psi \not \equiv -\infty$ and $\phi \not \equiv \infty$,
is Newton differentiable as a function from $L^p(\Omega)$ to $H_0^1(\Omega) \cap L^q(\Omega)$
whenever $\max(1,d/2) < p \leq \infty$ and $1 \leq q < \infty$. 
For the precise statements of these two main results (involving 
all assumptions and also the definitions of the respective Newton derivatives),
we refer the reader to \cref{sec:4,sec:6}.

Before we begin with our analysis, 
let us briefly put our work into perspective. 
For the sake of brevity, we focus on
infinite-dimensional problems. For additional information on the finite-dimensional setting,
we refer the reader to
 \cite{Facchinei2003,Izmailov2014,Mifflin1977,Mifflin1977-2,OutrataBook1998} and the references therein.

In optimization and optimal control, 
the concept of Newton differentiability (a.k.a.\ semismoothness)
plays a crucial role because it characterizes 
nonsmooth functions that 
are well behaved to such a degree that 
classical algorithms for the numerical solution of 
smooth minimization problems and operator equations become applicable.
See, e.g., \cite{Bartels2015,ChenNashedQi2000,Hintermueller2002,Mifflin1977,SchrammZowe1992,Ulbrich2011} 
for various contributions on this topic. Most prominently, 
in the presence of Newton differentiability, 
it is possible to generalize the classical Newton algorithm 
to the so-called \emph{semismooth Newton method}
which, under suitable invertibility assumptions 
on the involved generalized derivatives,
makes it possible
to solve nonsmooth operator equations 
with superlinear convergence speed.
In practice, the main difficulty in applying such a semismooth Newton method 
is proving that the nonsmooth operators under consideration 
indeed possess a Newton derivative.
In fact, establishing Newton differentiability results for infinite-dimensional functions 
is so hard that there are only a handful of results available on this topic. 
The one that is most commonly used in applications is 
the fact that nonsmooth superposition operators 
\mbox{$F\colon L^p(\Omega) \to L^q(\Omega)$,} $v \mapsto f(v)$, 
that are
induced by a globally Lipschitz continuous and Newton differentiable 
function $f\colon \R \to \R$, are Newton differentiable 
when $p$ and $q$ satisfy $1\leq q < p\leq \infty$ (a so-called norm gap)
and the Newton derivative of $F$ is constructed from that of $f$ by means 
of a pointwise-a.e.\ selection. 
We refer to, e.g.,
\cite{Reyes2005,Hintermueller2008,Roesch2011,Schiela2008,Ulbrich2011},
where this Newton differentiability result 
is used for the numerical solution of optimal control problems with $L^2$-box-constraints, 
and also the recent \cite{BrokateUlbrich2022}. 
For functions that are not of superposition type (or, more precisely, 
for which superposition operators are not the sole source of nonsmoothness), 
Newton differentiability results are very scarce. To the best of our knowledge,
there are only two papers containing genuine contributions on this topic. The first one, 
\cite{Brokate2020}, establishes the Newton differentiability of the solution mapping 
of a prototypical rate-independent evolution variational inequality---the so-called scalar stop operator---%
by means of a semi-explicit solution formula;
see \cite[Theorem 7.15]{Brokate2020}. The second one is our previous 
contribution \cite{ChristofWachsmuth2023},
which
demonstrates by means of pointwise convexity properties that 
solution maps of unilateral obstacle problems are Newton differentiable as 
functions between suitable Lebesgue spaces when the 
strong-weak Bouligand differential is used as a generalized derivative; 
see \cite[Theorem~2.12, Corollary~3.9]{ChristofWachsmuth2023}. 
Compare also again with \cite{BrokateUlbrich2022} in this context. 
Note that all of these existing contributions have in common that they rely
on pointwise properties of the considered nonsmooth operators---either arising
from a superposition structure or a certain curvature behavior---and 
thus only yield Newton differentiability results in Lebesgue spaces. 

In the present paper, we demonstrate how this limitation can be overcome. 
More precisely, we show how the fact that the set $K$ in \eqref{def:K} is polyhedric
(in the sense of \cite{Haraux1977})
can be exploited to prove that the solution operator of the unilateral obstacle problem 
is Newton differentiable as a function with values in the energy space $H_0^1(\Omega)$;
see \cref{th:semismoothH01unilateral}. We moreover demonstrate that this 
refined Newton differentiability result for the unilateral case 
and properties of projections onto angled subspaces in $H^{-1}(\Omega)$
make it possible to establish
the Newton differentiability of the solution map of the bilateral obstacle problem 
as a function from $L^p(\Omega)$ to $H_0^1(\Omega) \cap L^q(\Omega)$ 
for all $\max(1,d/2) < p \leq \infty$ and $1 \leq q < \infty$;
see \cref{th:semismoothH01bilateral}. We would like to point out that the latter of these two main results
is completely beyond the scope of the aforementioned pointwise approaches for proving 
Newton differentiability properties as the solution map of the bilateral obstacle problem 
neither has a superposition structure nor possesses pointwise-a.e.\ curvature properties 
that could be used to resort to the scalar situation. 
At least to the best of our knowledge, 
\cref{th:semismoothH01unilateral}, \cref{th:semismoothH01bilateral},
and the arguments that we employ for the proofs of these main results
are new and 
have no analogue in the literature. 
We remark that we restrict our attention to the classical obstacle problem 
\eqref{eq:VI} in our analysis mainly for simplicity and to avoid obscuring the basic ideas of our approach with technicalities. 
We expect that the techniques that we develop in the following 
can be extended rather straightforwardly to, e.g., Signorini-type VIs and thin-obstacle problems. 
Compare also with the general functional-analytic setting considered in \cite{ChristofWachsmuth2023} in this regard. 

We conclude this introduction with an overview of the content and the structure of 
the remainder of the paper. 
\Cref{sec:2,sec:3} are concerned with preliminaries.
Here, we fix the notation, introduce basic concepts, 
and collect known results on the obstacle problem \eqref{eq:VI}
that are needed for our analysis. 
In \cref{sec:4}, we prove the first main result of this work---the $H_0^1(\Omega)$-Newton differentiability
of the solution map $S\colon u \mapsto y$ of \eqref{eq:VI} in the unilateral case $\phi \equiv \infty$;
see \cref{th:semismoothH01unilateral}.
\Cref{sec:5} contains several auxiliary results on 
orthogonal projections and angled subspaces 
that are needed for the analysis of the 
bilateral case $\psi \not \equiv -\infty$ and $\phi \not \equiv \infty$.
The results in this section 
may also be of independent interest; see, e.g., 
 the formula for orthogonal projections onto direct sums
 in \cref{thm:some_invertibility}. 
 \Cref{sec:6} is concerned 
with the proof of our second main result---the Newton differentiability of the solution map
of the bilateral obstacle problem as a function from
 $L^p(\Omega)$ to $H_0^1(\Omega) \cap L^q(\Omega)$ 
for all $\max(1,d/2) < p \leq \infty$ and $1 \leq q < \infty$;
see \cref{th:semismoothH01bilateral}. 
In \cref{sec:7}, we conclude the paper with an explicit application example. 
Here, we show that \cref{th:semismoothH01bilateral} makes it possible to prove
that optimal control problems with $H_0^1(\Omega)$-controls and box-constraints 
can be solved by a semismooth Newton method in function space with superlinear convergence speed.
The results on this topic are also accompanied by 
numerical experiments.


\section{Notation and basic concepts}%
\label{sec:2}%
Throughout this paper,
we use the standard symbols 
$\|\cdot\|_X$, 
$\langle \cdot, \cdot \rangle_X$,
and 
$(\cdot, \cdot)_X$
to denote 
norms,
dual pairings, and 
inner products defined on a (real) vector space $X$,
 respectively. 
The space of linear and continuous functions
between two normed spaces $(X, \|\cdot\|_X)$
and $(Y, \|\cdot\|_Y)$ is denoted by $\LL(X,Y)$.
For the topological dual space, we write $X\dualspace := \LL(X,\R)$,
and for the orthogonal complement of a closed
subspace $W$ in a Hilbert space $H$, we write $W^\perp$.
The kernel and the range of a function $G \in \LL(X,Y)$
are denoted by $\ker(G)$ and $\range(G)$, respectively.
The identity map is denoted by $\Id$.
If a space
$(X, \|\cdot\|_X)$ is continuously embedded into a space $(Y, \|\cdot\|_Y)$, 
then we write $X \hookrightarrow Y$.
If the embedding is compact, 
then this is denoted by $X \compactly Y$.
Weak and strong convergences are denoted by 
$\weakly$ and $\to$, respectively. Given a sequence 
$\{G_n\} \subset \LL(X,Y)$, we say that 
$G_n$ converges in the weak operator topology (WOT) to $G \in \LL(X,Y)$,
in symbols $G_n \WOT G$, 
if $G_n h \weakly Gh$ holds in $Y$ for all $h\in X$. 

A mapping $F$ from $X$ to the power set of $Y$
is denoted by $F\colon X \rightrightarrows Y$.
 Recall that a function 
$F\colon X \to Y$ between two normed spaces $(X, \|\cdot\|_X)$
and $(Y, \|\cdot\|_Y)$
is called directionally differentiable if, 
for all $x, h \in X$, there exists $F'(x;h) \in Y$ (the directional derivative)
such that
\begin{align*}
\lim_{0 < t \to 0}\frac{F(x + t h) - F(x)}{t} = F'(x;h).
\end{align*}
If $F'(x;\cdot)$ is an element of $\LL(X,Y)$, then 
$F$ is called Gâteaux differentiable at $x$, $F'(x) := F'(x;\cdot) \in \LL(X,Y)$
is called the Gâteaux derivative of $F$ at $x$,
and $x$ is called a Gâteaux point of $F$.
The set of all Gâteaux points of a function $F\colon X \to Y$ 
is denoted by 
$\DD_F \subset X$. 
A directionally differentiable function $F\colon X \to Y$ satisfying 
\[
\{h_n\} \subset X,
\{t_n\} \subset (0, \infty),
h_n \to h,
t_n \to 0
\quad
\Rightarrow
\quad 
\frac{F(x + t_n h_n) - F(x)}{t_n} \to F'(x;h)
\]
for all $x,h \in X$ is called Hadamard directionally differentiable. 
Following \cite{Christof2018,Rauls2020}, 
we 
define the strong-weak Bouligand differential $\partial_B^{sw}F\colon X \rightrightarrows \LL(X,Y)$
of a function $F\colon X \to Y$ by
\begin{align*}
\partial_B^{sw}F(x)
:=
\set[\big]{
G \in \LL(X,Y)
\given \exists \{x_n\} \subset \DD_F \colon x_n \to x \text{ in }X,~ 
F'(x_n) \WOT G \text{ in } \LL(X,Y)}.
\end{align*}
A function $F\colon X \to Y$ 
equipped with a mapping
$\D F\colon X \rightrightarrows \LL(X,Y)$
satisfying $DF(x) \neq \emptyset$ for all $x \in X$
is called Newton differentiable if
\begin{equation}
\label{eq:NewtonDifDef}
\sup_{G \in \D F(x + h)} 
\frac{\|F(x + h) - F(x) - Gh\|_{Y}}{\|h\|_{X}} \to 0
\quad\text{ for } 0 <  \|h\|_{X} \to 0
\end{equation}
for all $x \in X$;
cf.\ \cite{ChenNashedQi2000,Hintermueller2002,Ulbrich2011}. In this case, $\D F$ is called a
Newton derivative of $F$. 

Given a closed convex nonempty subset $L$ of a space $(X,\|\cdot\|_X)$
and $x \in L$,
we denote by $\RR(x;L) := \R_+ (L - x)$ the radial cone of $L$ at $x$
and by $\TT(x;L) := \cl(\RR(x;L))$ the tangent cone of $L$
at $x$; cf.\ \cite[section 2.2.4]{BonnansShapiro2000}. 
Here, $\cl(\cdot)$ denotes a closure in $X$.

Given a nonempty open bounded set $\Omega \subset \R^d$, $d \in \N$,
we denote 
by  $L^p(\Omega)$, $H^k(\Omega)$, and $W^{k,p}(\Omega)$,
$k \in \N$, $p \in [1,\infty]$,
the standard real Lebesgue and Sobolev spaces on $\Omega$, respectively. 
The norms $\|\cdot\|_{L^p(\Omega)}$,
$\|\cdot\|_{H^k(\Omega)}$, and $\|\cdot\|_{W^{k,p}(\Omega)}$
are defined as usual; see \cite[chapter 5]{Attouch2006}. 
Analogously to \cite[Definition 5.1.4]{Attouch2006},
we define $H_0^1(\Omega)$
to be the closure of the set $C_c^\infty(\Omega)$ of smooth functions 
with compact support on $\Omega$ in $H^1(\Omega)$.
We equip the space $H_0^1(\Omega)$
with the energy norm of \eqref{eq:VI}, i.e.,
\begin{equation}
	\label{eq:norm_in_H01}
	\norm{v}_{H_0^1(\Omega)} 
	:=
	\left (
	\int_\Omega \abs{\nabla v}^2 \dd x \right )^{1/2}
	\qquad
	\forall v \in H_0^1(\Omega).
\end{equation}
Here, $|\cdot|$ denotes the Euclidean norm and $\nabla$ is the weak gradient. 
The associated inner product is denoted by \smash{$\innerp{\cdot}{\cdot}_{H_0^1(\Omega)}$}.
Due to Poincaré's inequality,
the norm $\norm{\cdot}_{H_0^1(\Omega)}$  is 
equivalent to $\norm{\cdot}_{H^1(\Omega)}$
on $H_0^1(\Omega)$. For the topological dual space of 
$(H_0^1(\Omega), \|\cdot\|_{H_0^1(\Omega)})$, we use the symbol \smash{$H^{-1}(\Omega)$}.
The Laplacian
$-\Delta \colon H_0^1(\Omega) \to H^{-1}(\Omega)$
is defined by
\begin{equation}
\label{eq:LaplaceRiesz}
	\dual{-\Delta v}{w}_{H_0^1(\Omega)}
	:=
	\innerp{v}{w}_{H_0^1(\Omega)}
	=
	\int_\Omega \nabla v \cdot \nabla w \dd x
	\qquad
	\forall v,w \in H_0^1(\Omega).
\end{equation}
Due to \eqref{eq:norm_in_H01} and \eqref{eq:LaplaceRiesz},
$(-\Delta)^{-1}$ is precisely the Riesz map
of $H_0^1(\Omega)$.
In particular, $-\Delta$ is an isometric isomorphism
between $H_0^1(\Omega)$ and $H^{-1}(\Omega)$.
This also implies that the inner product on $H^{-1}(\Omega)$
associated with the norm $\|\cdot\|_{H^{-1}(\Omega)}$ is given by 
\begin{equation}
\label{eq:H-1scalarproduct}
	\innerp{z_1}{z_2}_{H^{-1}(\Omega)}
	=
	\innerp{(-\Delta)^{-1} z_1}{(-\Delta)^{-1} z_2}_{H_0^1(\Omega)}
	\qquad
	\forall z_1, z_2 \in H^{-1}(\Omega).
\end{equation}
For the spaces of continuous functions 
on $\Omega$ and $\bar\Omega := \mathrm{cl}(\Omega)$, we use
the symbols $C(\Omega)$ and $C(\bar\Omega)$, respectively. We further define
$C_0(\Omega):=\{v \in C(\bar\Omega) \mid v=0 \text{ on }\partial \Omega\}$, where $\partial \Omega$ denotes the 
boundary of $\Omega$.
Recall that $C(\bar\Omega)$ and $C_0(\Omega)$ are Banach when endowed 
with the maximum norm $\|\cdot\|_\infty$. 
Given an element $v$ of $C(\Omega)$, 
we denote by $\{v = c\}$, $c \in \R$, the level set $\{x \in \Omega \mid v(x) = c\}$.
Analogous abbreviations are also used for sublevel, superlevel, and coincidence sets.
For the indicator function of a set $D \subset \Omega$,
we use the notation $\mathds{1}_{D}\colon \Omega \to \{0,1\}$,
and for the Euclidean distance to $D$,
we write $\dist(\cdot, D)$.
Note that
further symbols etc.\ are introduced in the following sections wherever necessary. 
This additional notation is clarified on its first appearance. 

\section{Preliminaries on the obstacle problem}
\label{sec:3}
In this section, we collect basic results on 
the obstacle problem \eqref{eq:VI}.
Throughout \cref{sec:3},
we utilize the following assumptions. 

\begin{assumption}[standing assumptions for \cref{sec:3}]~
\begin{itemize}
\item $\Omega \subset \R^d$, $d \in \N$, is a nonempty, open, and bounded set;
\item $\psi, \phi\colon \Omega \to [-\infty, \infty]$ are measurable functions such that $K \neq \emptyset$.
\end{itemize}
\end{assumption}

We begin with
a basic solution result for \eqref{eq:VI}.

\begin{proposition}[solvability and Lipschitz continuity]%
\label{prop:solvability}%
For all $u \in H^{-1}(\Omega)$,
the variational inequality \eqref{eq:VI} possesses a unique solution $S(u) := y \in H_0^1(\Omega)$.
The solution map $S\colon H^{-1}(\Omega) \to H_0^1(\Omega)$, $u \mapsto y$, 
of \eqref{eq:VI} satisfies
\begin{equation}
\label{eq:SLipschitzH01}
\norm*{
S(u_1) - S(u_2)
}_{H_0^1(\Omega)}
\leq
\|u_1 - u_2\|_{H^{-1}(\Omega)}\qquad \forall u_1, u_2 \in H^{-1}(\Omega).
\end{equation}
\end{proposition}

\begin{proof}
See \cite[Theorem~4-3.1]{Rodrigues1987} or \cite[Theorem II-2.1]{KinderlehrerStampacchia1980}. 
\end{proof}

Note that $S$ is---modulo the Riesz isomorphism---precisely the metric projection 
in $(H_0^1(\Omega), \|\cdot\|_{H_0^1(\Omega)})$ onto the set $K$. 
For $L^p$-inputs, we also
have the following Lipschitz continuity result for the mapping $S$.

\begin{proposition}[$L^p$-Lipschitz estimate]%
\label{prop:LpLipschitz}%
Let $p > \max(1, 2d/(d+2))$ and define
\begin{align}
\label{eq:Qpdef}
\QQ_p:=
\begin{cases}
[1,\infty] & \text{ if }  p > \frac{d}{2},
\\
\left [1, \parens*{\frac1p - \frac2d}^{-1}\right )& \text{ if }  p \leq \frac{d}{2},
\end{cases}
\end{align}
with the convention $0^{-1} := \infty$ in the case  $p = d/2$.
Then, for every $q \in \QQ_p$, there exists a constant $C > 0$ such that 
\begin{align}
\label{eq:LPLipschitzEstimate}
\norm*{
S(u_1) - S(u_2)
}_{L^q(\Omega)}
\leq
C
\|u_1 - u_2\|_{L^{p}(\Omega)}\qquad \forall u_1, u_2 \in L^p(\Omega).
\end{align}
\end{proposition}
\begin{proof}
This result follows from the Sobolev embeddings and 
classical truncation arguments.
See, e.g., \cite[Lemma~3.7]{ChristofWachsmuth2023}
where the assertion of the proposition is proved in the unilateral case.
The proof for the bilateral case is completely analogous. 
\end{proof}

Recall that $p > \max(1, 2d/(d+2))$  implies $L^p(\Omega) \compactly H^{-1}(\Omega)$
by the Sobolev embeddings.
Thus, $S(u_i) \in H_0^1(\Omega)$ in \cref{prop:LpLipschitz} is well defined for $u_i \in L^p(\Omega)$.
Note further that the functions $S(u_i)$ in \eqref{eq:LPLipschitzEstimate}
might not belong to $L^q(\Omega)$ due to the low regularity of
$\phi$ and $\psi$. Only the difference $S(u_1) - S(u_2)$ is known to be in $L^q(\Omega)$;
cf.\ \cite[Lemma 3.7]{ChristofWachsmuth2023} and the comments after \cite[Corollary 4.2]{ChristofWachsmuth2023}.
Next, we state a classical directional differentiability result for \eqref{eq:VI}
that goes back to \cite{Haraux1977,Mignot1976}.

\begin{proposition}[Hadamard directional differentiability]%
\label{prop:hadamarddifferentiability}%
The solution operator 
 $S\colon H^{-1}(\Omega) \to H_0^1(\Omega)$, $u \mapsto y$, 
 of \eqref{eq:VI}  is Hadamard directionally differentiable.
 The directional derivative $\delta := S'(u;h)$ of $S$ at
 a point  $u$ with state $y := S(u)$
 in direction $h$ is uniquely characterized by the variational inequality 
 \begin{equation}
 \label{eq:ddauxprob}
 \delta \in \KK(u),
 	 \qquad 
	 \dual*{-\Delta \delta   - h}{ z - \delta }_{H_0^1(\Omega)} \geq 0
	 \qquad\forall z \in \KK(u).
 \end{equation}
 Here,
 $\KK(u) := \TT(y;K) \cap \ker(\Delta y + u)$
is the so-called \emph{critical cone} associated with $u$.
\end{proposition}

\begin{proof}
See \cite[Corollary 3.4.3]{ChristofPhd2018}
and also the classical works \cite{Haraux1977,Mignot1976}.
\end{proof}

\begin{remark}
The proof of \cref{prop:hadamarddifferentiability} 
relies on the so-called \emph{polyhedricity} of the set $K$,
i.e., the fact that, for all $u \in H^{-1}(\Omega)$ with
state $y  \in H_0^1(\Omega)$, one has
\begin{equation}
\label{eq:polyhedricity}
\cl_{H_0^1(\Omega)}\left ( \RR(y;K) \cap \ker(\Delta y + u) \right )
= \TT(y;K) \cap \ker(\Delta y + u).
\end{equation}

A tangible characterization of the critical cone $\KK(u) = \TT(y;K) \cap \ker(\Delta y + u)$
appearing on the right-hand side of \eqref{eq:polyhedricity}
can be obtained by using the notion of quasi-everywhere
and the fine-support of the multiplier $\Delta y + u \in H^{-1}(\Omega)$.
For details on this topic, we refer to \cite[section~2.2]{Rauls2020},
\cite{Haraux1977,Mignot1976,Wachsmuth2016:2}, 
 and the references therein.
\end{remark}

Using \eqref{eq:ddauxprob},
one can check that $u$ is an element of the set $\DD_S$ of Gâteaux points of
$S\colon H^{-1}(\Omega) \to H_0^1(\Omega)$
if and only if the critical cone $\KK(u)$ is a subspace.
From the Lipschitz estimate \eqref{eq:SLipschitzH01}, 
we obtain that there are indeed many such $u \in H^{-1}(\Omega)$.

\begin{theorem}[Rademacher's theorem]%
\label{th:infRademacher}%
The set $\DD_S$ 
is dense in $H^{-1}(\Omega)$.
\end{theorem}
\begin{proof}
See \cite[Proposition~1.3, Remark~1.3]{Thibault1982} and the references therein. 
\end{proof}

In tandem with \eqref{eq:SLipschitzH01},
the density of $\DD_S$  in $H^{-1}(\Omega)$ 
implies the following result for the 
strong-weak Bouligand differential 
$\partial_B^{sw}S\colon H^{-1}(\Omega) \rightrightarrows\LL(H^{-1}(\Omega), H_0^1(\Omega))$
of the solution operator $S\colon H^{-1}(\Omega) \to H_0^1(\Omega)$
of \eqref{eq:VI}. 

\begin{lemma}[{boundedness and nonemptyness of $\partial_B^{sw}S(u)$}]%
The set $\partial_B^{sw}S(u)$ is nonempty for all $u \in H^{-1}(\Omega)$
and it holds
\begin{equation}
	\label{eq:estimate_pbsw}
	\norm{G}_{\LL(H^{-1}(\Omega),H_0^1(\Omega))} \le 1
	\quad
	\forall G \in \partial_B^{sw}S(u)\quad \forall u \in H^{-1}(\Omega).
\end{equation}
\end{lemma}

\begin{proof}
From \eqref{eq:SLipschitzH01}, it follows that 
$\norm{S'(u)}_{\LL(H^{-1}(\Omega),H_0^1(\Omega))} \le 1$
holds for all $u \in \DD_S$. In combination with the Banach--Alaoglu theorem 
for the WOT, \cref{th:infRademacher}, and the definition of  $\partial_B^{sw}S$,
this yields $\partial_B^{sw}S(u) \neq \emptyset$ for all $u \in H^{-1}(\Omega)$;
see \cite[Theorem~2.5, Corollary~2.6]{ChristofWachsmuth2023}. 
To establish \eqref{eq:estimate_pbsw},
it suffices to note that, 
for all $u,h \in H^{-1}(\Omega)$ and all 
$G \in \partial_B^{sw}S(u)$
with associated Gâteaux points $\{u_n\}$ as in the definition of $\partial_B^{sw}S(u)$, we have 
\begin{align*}
	\norm{G h}_{H_0^1(\Omega)}
	\le
	\liminf_{n \to \infty} \norm{S'(u_n) h}_{H_0^1(\Omega)}
	\le
	\norm{h}_{H^{-1}(\Omega)}
\end{align*}
by the weak sequential lower semicontinuity of the norm. 
This completes the proof.~
\end{proof}

The nonemptyness of $\partial_B^{sw}S(u)$ makes 
the strong-weak Bouligand differential a suitable candidate 
for a Newton derivative of $S$; cf.\ 
\cref{th:LPsemismooth,th:semismoothH01unilateral,th:semismoothH01bilateral} below.
As a final preparation for our analysis, 
we note the following.

\begin{lemma}[norm of  derivatives]%
\label{lem:energy}%
Let $u, h \in H^{-1}(\Omega)$ be given. 
\begin{enumerate}
\item\label{lemma:energy:i} The derivative $\delta := S'(u;h)$ of $S$ at $u$ in 
direction $h$ satisfies
\begin{align}
\label{eq:energy1}
	\norm{\delta}_{H_0^1(\Omega)}^2 = \dual{h}{\delta }_{H_0^1(\Omega)}.
\end{align}
\item\label{lemma:energy:ii} For all  $G \in \partial_B^{sw}S(u)$, the function $\delta := Gh$ satisfies 
\begin{align*}
	\norm{\delta}_{H_0^1(\Omega)}^2 \leq  \dual*{h}{\delta}_{H_0^1(\Omega)}.
\end{align*}
\end{enumerate}
\end{lemma}
\begin{proof}
Since $\KK(u)$ is a cone, 
we may choose $z = 0$ and $z = 2\delta$ in \eqref{eq:ddauxprob}. This yields 
$\dual*{-\Delta \delta   - h}{\delta }_{H_0^1(\Omega)} = 0$ and, thus, \eqref{eq:energy1}.
Suppose now that an operator
$G \in \partial_B^{sw}S(u)$ is given and that $\{u_n\} \subset \DD_S$ is a sequence of 
Gâteaux points for $G$ as in the definition of $\partial_B^{sw}S(u)$. In this situation, 
we have $S'(u_n)h \weakly Gh$ in $H_0^1(\Omega)$ and, 
by the weak sequential lower semicontinuity of the norm
and  \ref{lemma:energy:i}, 
\begin{align*}
\norm{G h}_{H_0^1(\Omega)}^2
\leq
\liminf_{n \to \infty}
\norm{S'(u_n) h}_{H_0^1(\Omega)}^2
=
\liminf_{n \to \infty}
 \dual*{h}{S'(u_n)h }_{H_0^1(\Omega)}
 =
 \dual*{h}{G h}_{H_0^1(\Omega)}.
\end{align*}
This establishes the assertion of \ref{lemma:energy:ii} and completes the proof. 
\end{proof}

\begin{remark}
\Cref{lem:energy} relies crucially 
on the polyhedricity of the set $K$.
Indeed, for nonpolyhedric sets, the variational inequality \eqref{eq:ddauxprob}
typically contains an additional curvature term that causes 
\eqref{eq:energy1} to be invalid and whose behavior for varying $u$
is hard to control; see \cite[Theorem 1.4.1(ii)]{ChristofPhd2018}
and \cite[Theorem 4.1]{ChristofWachsmuth2020}.
\end{remark}

\section{Energy space Newton differentiability for the unilateral case}
\label{sec:4}
With the preliminaries in place, 
we are in the position to turn our attention 
to the proof of the first main result of this paper---the 
Newton differentiability of the solution operator of the unilateral 
obstacle problem with values in the energy space $H_0^1(\Omega)$.

\begin{assumption}[standing assumptions for \cref{sec:4}]~
\begin{itemize}
\item $\Omega \subset \R^d$, $d \in \N$, is a nonempty, open, and bounded set;
\item $\psi\colon \Omega \to [-\infty, \infty]$ is measurable, $\phi \equiv \infty$,
and $K \neq \emptyset$.
\end{itemize}
\end{assumption}

The point of departure for our analysis is the following result 
that has recently been obtained in \cite{ChristofWachsmuth2023}.

\begin{theorem}[$L^2$-Newton differentiability for the unilateral obstacle problem]%
\label{th:LPsemismooth}%
The solution map $S$
of \eqref{eq:VI}
is Newton differentiable 
as a function
$S\colon L^p(\Omega) \to L^2(\Omega)$ 
for all $\max(1, 2d/(d+2)) < p \leq \infty$
with Newton derivative $DS := \partial_B^{sw}S$.
\end{theorem}
\begin{proof}
See \cite[Corollaries 3.9, 4.2]{ChristofWachsmuth2023}. 
\end{proof}

\begin{remark}
The function 
$S$ is even Newton differentiable with values in $L^q(\Omega)$
for all $q \in \QQ_p \setminus \{\infty\}$ where $\QQ_p$ is the set in \eqref{eq:Qpdef}.
Stating this generalized Newton differentiability result requires some care though
since $S$ might not map into $L^q(\Omega)$
due to the (possibly) low regularity of $\psi$.
The interested reader is referred to
\cite{ChristofWachsmuth2023}. 
\end{remark}

To establish the Newton differentiability of $S$ as a function with 
values in $H_0^1(\Omega)$, we combine 
\cref{th:LPsemismooth}
with \cref{lem:energy}. 

\begin{theorem}[$H_0^1$-Newton differentiability for the unilateral obstacle problem]%
\label{th:semismoothH01unilateral}%
The solution map $S$
of \eqref{eq:VI}
is Newton differentiable 
as a function
$S\colon L^p(\Omega) \to H_0^1(\Omega)$ 
for all $\max(1, 2d/(d+2)) < p \leq \infty$
with Newton derivative $DS := \partial_B^{sw}S$.
\end{theorem}
\begin{proof}
We assume without loss of generality that $p < \infty$. 
(If the claim is proved for such $p$, then the assertion 
for $p=\infty$ follows immediately by H\"older's inequality.)
In order to prove the theorem, 
it suffices to show
that, for all $u \in L^p(\Omega)$ and all
sequences $\{t_n\} \subset (0, \infty)$, $\{h_n\} \subset L^p(\Omega)$, and
$\{G_n\} \subset \LL(H^{-1}(\Omega) , H_0^1(\Omega))$
satisfying $t_n \to 0$, 
$\norm{h_n}_{L^p(\Omega)} = 1$, and $G_n \in \partial_B^{sw}S(u + t_n h_n)$,
we have
\begin{equation}
\label{eq:contra}
\frac{\norm{S(u + t_n h_n) - S(u) - G_n t_n h_n}_{H_0^1(\Omega)}}{\norm{t_n h_n}_{L^p(\Omega)}}
\to
0;
\end{equation}
cf.\ the definition of Newton differentiability in \eqref{eq:NewtonDifDef}.
Due to the boundedness of $\{h_n\}$ in $L^p(\Omega)$
and the compact embedding $L^p(\Omega) \compactly H^{-1}(\Omega)$,
every subsequence of $\{h_n\}$
possesses a further subsequence (not relabeled)
such that
$h_n \weakly h$ in $L^p(\Omega)$ and
$h_n \to h$ in $H^{-1}(\Omega)$
for some $h \in L^p(\Omega)$.
From \cref{th:LPsemismooth},
we further get that 
\begin{align*}
 \frac{\norm{S(u + t_n h_n) - S(u) - G_n t_n h_n}_{L^2(\Omega)}}{\|t_n h_n\|_{L^p(\Omega)}} \to 0,
\end{align*}
and from \eqref{eq:SLipschitzH01} and \eqref{eq:estimate_pbsw}, it follows that there exists a constant $C>0$ satisfying
\begin{align*}
	\norm*{ \frac{S(u + t_n h_n) - S(u)}{t_n} - G_n h_n }_{H_0^1(\Omega)} \leq C.
\end{align*}
Combining these two observations yields
\begin{equation}
\label{eq:randomeq2636}
\frac{S(u + t_n h_n) - S(u)}{t_n} - G_n h_n  \weakly 0 \text{ in } H_0^1(\Omega).
\end{equation}
Due to the convergence $h_n \to h$ in $H^{-1}(\Omega)$ 
and \cref{prop:hadamarddifferentiability},
we know that the first term on the left-hand side of \eqref{eq:randomeq2636} 
converges strongly in $H_0^1(\Omega)$ to $S'(u;h)$.  
Thus,
\begin{equation}
\label{eq:randomeq373zdb48}
G_n h_n  \weakly  S'(u;h) \text{ in } H_0^1(\Omega).
\end{equation}
Define $\delta_n := G_n h_n $, $\delta := S'(u;h)$.
From \cref{lem:energy}, 
we obtain that 
\[
	\norm{\delta_n}_{H_0^1(\Omega)}^2
\leq 
\dual*{h_n}{ \delta_n}_{H_0^1(\Omega)}
~\forall n \in \N
\qquad
\text{and}
\qquad
	\norm{\delta}_{H_0^1(\Omega)}^2
=
\dual*{h }{ \delta}_{H_0^1(\Omega)}.
\]
Due to
the convergence $\delta_n \weakly \delta$ in $H_0^1(\Omega)$
in \eqref{eq:randomeq373zdb48} and $h_n \to h$ in $H^{-1}(\Omega)$,
it follows
\begin{align*}
	\norm{\delta - \delta_n}_{H_0^1(\Omega)}^2
	&=
	\norm{\delta_n}_{H_0^1(\Omega)}^2
	-
	\norm{\delta}_{H_0^1(\Omega)}^2
	+
	2\innerp{\delta}{\delta - \delta_n}_{H_0^1(\Omega)}
	\\
	&\le
	\dual*{h_n}{ \delta_n}_{H_0^1(\Omega)}
	-
	\dual*{h }{ \delta}_{H_0^1(\Omega)}
	+
	2\innerp{\delta}{\delta - \delta_n}_{H_0^1(\Omega)}
	\to
	0
	.
\end{align*}
Thus,
$\delta_n \to \delta$ in $H_0^1(\Omega)$. In summary, we have now shown that 
\[
\frac{S(u + t_n h_n) - S(u)}{t_n} - G_n h_n
=
\frac{S(u + t_n h_n) - S(u)}{t_n}
- S'(u;h)
+ \delta - \delta_n
\to 0
\]
holds in $H_0^1(\Omega)$
for the subsequence $\{h_n\}$ selected after \eqref{eq:contra}.
A standard subsequence-subsequence argument
now shows that \eqref{eq:contra} holds. This completes the proof.
\end{proof}

\begin{remark}
Explicit formulas for certain elements of the strong-weak Bouligand differential
$\partial_B^{sw}S$ of 
the solution map $S$ of the unilateral obstacle problem 
can be found in \cite[Proposition~2.11(i), Theorem~4.3]{Rauls2020}.
Under mild additional assumptions,
a precise characterization of the entire differential $\partial_B^{sw}S$ has been 
obtained in \cite[Theorem 5.6]{Rauls2020}. For a tangible corollary of the results of \cite{Rauls2020}
that is suitable for practical applications, we refer the reader to 
\cref{def:GenDerUni,lem:generalderivativeUni} in \cref{sec:6}.
\end{remark}

\section{Auxiliary results on angled subspaces and orthogonal projections}
\label{sec:5}
The purpose of this section is to establish 
several auxiliary results on angled subspaces and orthogonal projections
that are needed for the analysis of the bilateral obstacle problem in \cref{sec:6}.
For the sake of reusability, we consider a general setting.

\begin{assumption}[standing assumptions for \cref{sec:5}]~
\begin{itemize}
\item $H$ is a real Hilbert space with
closed unit ball $B_H := \{x \in H \mid \|x\|_H \leq 1\}$;
\item $W_1, W_2 \subset H$ are closed subspaces of $H$.
\end{itemize}
\end{assumption}

We start by defining
the minimal angle between $W_1$ and $W_2$;
see \cite{Deutsch1995} and the references therein for more information on this concept.
\begin{definition}[minimal angle]%
	\label{def:angle}%
	The minimal angle
	$\alpha_0(W_1, W_2)$
	between $W_1$ and $W_2$ in $H$
	is defined to be
	the angle in $[0,\pi/2]$
	whose cosine is given by
	\begin{equation*}
		c_0(W_1, W_2)
		:=
		\sup\set[\big]{
			\innerp{x_1}{x_2}_H
			\given
			x_1 \in W_1 \cap B_H,
			x_2 \in W_2 \cap B_H
		}
		\in [0,1].
	\end{equation*}
\end{definition}

It is clear that
$\alpha_0(\cdot,\cdot)$
and
$c_0(\cdot,\cdot)$
are symmetric.
Via a simple scaling argument, it is immediate that
$c_0 = c_0(W_1,W_2)$
is the smallest constant in $[0,1]$ for which
the following sharpened Cauchy--Bunyakovsky--Schwarz inequality holds true:
\begin{align*}
	\abs{\innerp{x_1}{x_2}_H}
	\le
	c_0 \norm{x_1}_H \norm{x_2}_H
	\qquad\forall x_1 \in W_1, x_2 \in W_2.
\end{align*}
Next, we recall \cite[Lemma~10~(3)]{Deutsch1995}.
For convenience, we also include its short proof.
\begin{lemma}%
	\label{lem:projnormH}%
	Denote by $P_i$ the orthogonal projection onto $W_i$, $i=1,2$.
	Then,
	\begin{align*}
		\norm{P_1 P_2}_{\LL(H,H)}
		=
		\norm{P_2 P_1}_{\LL(H,H)}
		=
		c_0(W_1, W_2).
	\end{align*}
\end{lemma} 
\begin{proof}
	We have
	\begin{align*}
		c_0(W_1, W_2)
		&=
		\sup\set[\big]{
			\innerp{x_1}{x_2}_H
			\given
			x_1 \in W_1 \cap B_H, x_2 \in W_2 \cap B_H
		}
		\\
		&=
		\sup\set[\big]{
			\innerp{P_1 x_1}{P_2 x_2}_H
			\given
			x_1 \in B_H, x_2 \in B_H
		}
		\\
		&=
		\sup\set[\big]{
			\innerp{x_1}{P_1 P_2 x_2}_H
			\given
			x_1 \in B_H, x_2 \in B_H
		}
		=
		\norm{P_1 P_2}_{\LL(H,H)}
		.
	\end{align*}
	The symmetry of $c_0(\cdot, \cdot)$
	provides the assertion for $P_2 P_1$.
\end{proof}

The next lemma shows
that the sine of the minimal angle
has a geometric meaning too.
It is a quantitative version of
\cite[Theorem~12 (1)$\Leftrightarrow$(3)]{Deutsch1995}.
\begin{lemma}%
	\label{lem:sine_angle}%
	Denote by $s_0 = s_0(W_1,W_2)$ the largest number in $[0,1]$
	satisfying
	\begin{align*}
		\norm{x_1 + x_2}_H
		\ge
		s_0 \norm{x_1}_H
		\qquad
		\forall x_1 \in W_1, x_2 \in W_2.
	\end{align*}
	Then it holds that
	\begin{equation}
	\label{eq:pythagoras}
		s_0(W_1, W_2)^2 + c_0(W_1, W_2)^2 = 1.
	\end{equation}
	In particular, $s_0(W_1, W_2)$ is the sine of the minimal angle
	$\alpha_0(W_1,W_2)$.
\end{lemma}
\begin{proof}
	We set $c_0 := c_0(W_1, W_2)$.
	For all $x_1 \in W_1$, $x_2 \in W_2$,
	we have
	\begin{align*}
		\norm{x_1 + x_2}_H^2
		&=
		\norm{x_1}_H^2 + 2 \innerp{x_1}{x_2}_H + \norm{x_2}_H^2
		\ge
		\norm{x_1}_H^2 - 2 c_0 \norm{x_1}_H \norm{x_2}_H + \norm{x_2}_H^2
		\\
		&\ge
		\norm{x_1}_H^2 - \norm{x_1}_H^2 - c_0^2 \norm{x_2}_H^2 + \norm{x_2}_H^2
		=
		(1 - c_0^2) \norm{x_2}_H^2,
	\end{align*}
	hence $s_0 \ge \sqrt{1 - c_0^2}$.
	Similarly,
	for $x_1 \in W_1 \cap B_H$,
	$x_2 \in W_2 \cap B_H$,
	and $t > 0$,
	we find
	\begin{align*}
		2 \innerp{x_1}{x_2}_H
		&=
		2 \innerp{t^{-1/2} x_1}{t^{1/2} x_2}_H
		=
		\norm{t^{-1/2} x_1}_H^2 + \norm{t^{1/2} x_2}_H^2
		-\norm{t^{-1/2} x_1 - t^{1/2} x_2}_H^2
		\\
		&\le
		t^{-1} \norm{x_1}_H^2 + t \norm{x_2}_H^2
		-s_0^2 t^{-1} \norm{x_1}_H^2
		\leq
		(1 - s_0^2) t^{-1} + t.
	\end{align*}
	Taking
	the supremum over all $x_i \in W_i \cap B_H$, $i=1,2$,
	and
	the infimum over all $t > 0$
	shows that $c_0 \le \sqrt{1 - s_0^2}$.
	This establishes the claim.
\end{proof}

Note that the symmetry
of
$s_0(\cdot, \cdot)$
is not obvious from the definition,
but it holds due to the symmetry of $c_0(\cdot, \cdot)$.
Finally,
we show the invertibility of a certain
operator
that appears in \cref{sec:6}.
\begin{theorem}%
	\label{thm:some_invertibility}%
	Let $P_i$ 
	denote the orthogonal projection onto $W_i$, $i=1,2$.
	Suppose that $W_1$ and $W_2$ enclose a positive angle, i.e., 
	$c_0(W_1,W_2) < 1$.
	Then the operator
	\begin{equation}
	\label{eq:R1def}
		R_1 :=
		\begin{pmatrix}
			\Id & P_1 \\
			P_2 & \Id
		\end{pmatrix}\in \LL( H^2, H^2)
	\end{equation}
	is continuously invertible with
	$\norm{R_1^{-1}}_{\LL(H^2,H^2)} \le 4/(1-c_0(W_1,W_2)) $,
	the subspace $W_1 + W_2$ is closed, and the
	orthogonal projection $P$ onto $W_1 + W_2$ is given by
	\begin{equation}
	\label{eq:Pdef}
		P =
		\begin{pmatrix}
			\Id & \Id
		\end{pmatrix}
		R_1^{-1}
		\begin{pmatrix}
			P_1 \\ P_2
		\end{pmatrix}.
	\end{equation}
\end{theorem}
\begin{proof}
	By a standard Schur complement approach, we
	find the factorization
	\begin{align*}
		R_1
		=
		\begin{pmatrix}
			\Id & P_1 \\
			P_2 & \Id
		\end{pmatrix}
		=
		\begin{pmatrix}
			\Id & 0 \\
			P_2 & \Id
		\end{pmatrix}
		\begin{pmatrix}
			\Id & 0 \\
			0 & \Id - P_2 P_1
		\end{pmatrix}
		\begin{pmatrix}
			\Id & P_1 \\
			0 & \Id
		\end{pmatrix}
		.
	\end{align*}
	Since $\norm{P_2 P_1}_{\LL(H,H)} = c_0(W_1,W_2) < 1$ holds by \cref{lem:projnormH},
	it follows from the Neumann series that the operator 
	$\Id - P_2 P_1$
	is invertible with
	$\norm{(\Id - P_2 P_1)^{-1}}_{\LL(H,H)} \le 1/(1 - c_0(W_1,W_2))$.
	This implies that $R_1$ is continuously invertible too with
	\begin{equation}
	\label{eq:R1-1def}
		R_1^{-1}
		=
		\begin{pmatrix}
			\Id & -P_1 \\
			0 & \Id
		\end{pmatrix}
		\begin{pmatrix}
			\Id & 0 \\
			0 & (\Id - P_2 P_1)^{-1}
		\end{pmatrix}
		\begin{pmatrix}
			\Id & 0 \\
			-P_2 & \Id
		\end{pmatrix}
	\end{equation}
	and
\begin{align*}
		\norm{R_1^{-1}}_{\LL(H^2,H^2)}
		\le
		2 \cdot \norm{(\Id - P_2 P_1)^{-1}}_{\LL(H,H)} \cdot 2
		\le
		\frac{4}{1 - c_0(W_1,W_2)}.
\end{align*}
	It remains to check that $W_1 + W_2$ is closed and that 
	the map $P \in \LL(H,H)$ in \eqref{eq:Pdef} is the orthogonal projection
	onto $W_1 + W_2$.
	By combining \eqref{eq:Pdef} with \eqref{eq:R1-1def},
	we get that
	\begin{equation}
	\label{eq:Preform}
	\begin{aligned}
		P
		&=
		\begin{pmatrix}
			\Id & \Id - P_1
		\end{pmatrix}
		\begin{pmatrix}
			\Id & 0 \\
			0 & (\Id - P_2 P_1)^{-1}
		\end{pmatrix}
		\begin{pmatrix}
			P_1 \\ P_2 - P_2 P_1
		\end{pmatrix}
		\\
		&=
		P_1 + (\Id - P_1) (\Id - P_2 P_1)^{-1} P_2 (\Id - P_1).
	\end{aligned}
	\end{equation}
	Since $P_i = P_i^2$ and $P_i(\Id - P_i) = (\Id - P_i)P_i = 0$ for $i=1,2$, \eqref{eq:Preform} yields that 
	\begin{align*}
		P^2
		&=
		P_1
		+ (\Id - P_1) (\Id - P_2 P_1)^{-1} P_2 (\Id - P_1)
		(\Id - P_2 P_1)^{-1} P_2 (\Id - P_1)
		\\
		&=
		P_1
		+ (\Id - P_1) (\Id - P_2 P_1)^{-1} P_2 (\Id - P_2 P_1)
		(\Id - P_2 P_1)^{-1} P_2 (\Id - P_1)
		\\
		&= P
		.
	\end{align*}
	Hence, $P$ is idempotent.
	Next, we note that
	\begin{equation}
	\label{eq:Neumannformula}
		(\Id - P_2 P_1)^{-1}
		=
		\sum_{k=0}^\infty (P_2 P_1)^k
		=
		\Id + P_2 (\Id - P_1 P_2)^{-1} P_1. 
	\end{equation}
	Plugging this series expansion into \eqref{eq:Preform} gives
	\begin{equation}
	\label{eq:randomeq22737}
			P = 
			P_1 + (\Id - P_1) \sum_{k = 0}^\infty (P_2 P_1)^k P_2 (\Id - P_1)
			.
	\end{equation}
	As the orthogonal projections $P_1$ and $P_2$ are self-adjoint, 
	\eqref{eq:randomeq22737} implies that $P$ is self-adjoint too.
	Thus, $P$ is the orthogonal projection onto $\range(P)$ by \cite[Lemma 4]{Deutsch1995}.
	In particular, $\range(P)$ is closed. 
	It remains to show that $\range(P) = W_1 + W_2$.
	From \eqref{eq:Preform}, we obtain that  $P P_1 = P_1$ holds
	and that
	\begin{align*}
		P P_2
		&=
		P_1 P_2 + (\Id - P_1) (\Id - P_2 P_1)^{-1} P_2 (\Id - P_1) P_2
		\\
		&=
		P_1 P_2 + (\Id - P_1) (\Id - P_2 P_1)^{-1} (\Id - P_2 P_1) P_2
		=
		P_2
		.
	\end{align*}
	Hence, $W_1 + W_2 \subset \range(P)$.
	From \eqref{eq:Preform} and \eqref{eq:Neumannformula}, we further obtain that
	\begin{align*}
		P
		&=
		P_1 - P_1 (\Id - P_2 P_1)^{-1} P_2 (\Id - P_1)
		+
		(\Id - P_2 P_1)^{-1} P_2 (\Id - P_1)
		\\
		&=
		P_1 - P_1 (\Id - P_2 P_1)^{-1} P_2 (\Id - P_1)
		+
		\parens*{
			\Id + P_2 (\Id - P_1 P_2)^{-1} P_1
		} P_2 (\Id - P_1)
		\\
		&=
		P_1 - P_1 (\Id - P_2 P_1)^{-1} P_2 (\Id - P_1)
		+
		P_2
		\parens*{
			\Id + (\Id - P_1 P_2)^{-1} P_1 P_2
		} (\Id - P_1).
	\end{align*}
	This shows that $\range(P) \subset W_1 + W_2$.
	Thus, $\range(P) = W_1 + W_2$,
	the operator $P$ is the orthogonal projection onto $W_1 + W_2$,
	and $W_1 + W_2$ is a closed subspace.~
\end{proof}

\begin{remark}~
\begin{itemize}
\item The closedness of the set $W_1 + W_2$ in \cref{thm:some_invertibility}
can also be derived directly from the condition $c_0(W_1,W_2) < 1$; 
cf.\ \cite[Theorem 12]{Deutsch1995}. 
It is interesting to see
that it is also obtained as a byproduct when analyzing the formula \eqref{eq:Pdef}.
\item 
From \eqref{eq:randomeq22737}, we obtain the
series representation
\begin{align*}
P
&=
\sum_{k=0}^\infty
\Big(
P_1 (P_2 P_1)^k
+
(P_2P_1)^k P_2
-
(P_1P_2)^{k+1}
-
(P_2P_1)^{k+1}
\Big )
\\
&=
	\parens{ P_1 + P_2 }
	-
	\parens{ P_1 P_2 + P_2 P_1 }
	+
	\parens{ P_1 P_2 P_1 + P_2 P_1 P_2 }
	-
	\ldots
\end{align*}
for the orthogonal 
projection $P$ onto the sum $W_1 + W_2$ of two 
closed angled subspaces $W_i$, $i=1,2$, in a Hilbert space $H$ 
in terms of the orthogonal projections $P_i$, $i=1,2$,
onto the spaces $W_i$.
Note that, 
in contrast to the formulas in \eqref{eq:Preform} and \eqref{eq:randomeq22737},
this representation is symmetric in $P_1$ and $P_2$.
\end{itemize}
\end{remark}

\section{Energy space Newton differentiability for the bilateral case}
\label{sec:6}

We are now in the position to prove the 
Newton differentiability of the solution operator of \eqref{eq:VI}
in the case $\psi \not \equiv -\infty$ and $\phi \not \equiv \infty$.
For the analysis of this bilateral case, we require slightly more restrictive assumptions 
than in \cref{sec:4}. 

\begin{assumption}[standing assumptions for \cref{sec:6}]~
\label{ass:sec6}
\begin{itemize}
\item $\Omega \subset \R^d$, $d \in \N$, is a nonempty, open, and bounded set;
\item $p > \max(1,d/2)$ is a given, fixed exponent;
\item $\Omega$ is such that it holds 
\begin{equation}
\label{eq:regOmega}
v \in H_0^1(\Omega), -\Delta v \in L^p(\Omega)
\qquad \Rightarrow
\qquad v \in C_0(\Omega);
\end{equation}
\item $\psi$ and $\phi$ satisfy
$\psi, \phi \in C(\bar\Omega) \cap H^1(\Omega)$ and $\Delta\psi, \Delta\phi \in L^p(\Omega)$ (where $\Delta$ is understood distributionally), 
and there exists a number $c > 0$ such that  
\begin{equation}
\label{eq:constant_c}
\psi \leq \phi - c \text{ in } \Omega
\qquad
\text{and}
\qquad
\psi  \leq - c \leq c \leq \phi  \text{ on } \partial \Omega. 
\end{equation}
\end{itemize}
\end{assumption}

\begin{remark}\label{rem:sec6:LipschitzDomain}%
If $\Omega$ is a bounded Lipschitz domain, then 
\eqref{eq:regOmega} holds automatically;
see \cite[Theorem 1.1]{Rehberg2015}. For arbitrary nonempty, open, bounded sets, \eqref{eq:regOmega} may fail
(as one may easily check by means of the example
\mbox{$\Omega = \set{ x \in \R^d \given 0 < |x| < 1}$}, $d > 1$).
\end{remark}

For later reference,
we provide a small lemma.
\begin{lemma}%
	\label{lem:min_in_hnulleins}%
	Let $v \in H_0^1(\Omega)$ be given.
	Then, $\min(v, \phi), \max(v, \psi) \in H_0^1(\Omega)$.
\end{lemma}
\begin{proof}
	We only prove the claim for $\min(v, \phi)$.
	From Stampacchia's lemma,
	see, e.g., \cite[Theorem 5.8.2]{Attouch2006},
	we get that $\min(v, \phi) \in H^1(\Omega)$.
	By definition, there exists a sequence $\{v_n\} \subset C_c^\infty(\Omega)$
	with $v_n \to v$ in $H^1(\Omega)$.
	By using Stampacchia's lemma again,
	we get $\min(v_n, \phi) \in H^1(\Omega) \cap C_c(\Omega)$.
	A standard mollification argument provides us with
	$\min(v_n, \phi) \in H_0^1(\Omega)$.
	It is easy to check that
	$\min(v_n, \phi) \to \min(v,\phi)$ in $L^2(\Omega)$.
	Further, this sequence is bounded in $H^1(\Omega)$
	which yields
	$\min(v_n, \phi) \weakly \min(v,\phi)$ in $H^1(\Omega)$.
	Since $H_0^1(\Omega)$ is a closed subspace,
	we infer $\min(v,\phi) \in H_0^1(\Omega)$.
\end{proof}

Note that \cref{lem:min_in_hnulleins} implies in particular that $\min(\phi, \max(0,\psi)) \in H_0^1(\Omega)$ holds.
Since $\psi \leq \min(\phi, \max(0,\psi))  \leq \phi$, 
this shows that the admissible set $K$ 
is nonempty in the situation of \cref{ass:sec6} and that the results of
\cref{sec:3,sec:4} are applicable. 

 The basic idea of the following analysis is
 to establish the Newton differentiability of the 
 solution operator of the bilateral obstacle problem 
 by means of  the results for the unilateral case in \cref{sec:4} and an argumentation similar
 to that in the proof of the classical implicit function theorem.
 To pursue this approach, we require some additional notation.
 Henceforth, we denote by 
 $S_\psi$ the solution operator of the obstacle problem \eqref{eq:VI}
 with lower obstacle $\psi$ and upper obstacle $\infty$,
 by $S^\phi$ the solution operator of the obstacle problem \eqref{eq:VI}
 with lower obstacle $-\infty$ and upper obstacle $\phi$,
 and by \smash{$S_\psi^\phi$} the solution operator of the bilateral obstacle problem 
 with lower obstacle $\psi$ and upper obstacle $\phi$.  
 Note that all of these operators are well defined and globally 
 Lipschitz continuous as functions from $H^{-1}(\Omega)$ to $H_0^1(\Omega)$;
 see \cref{prop:solvability}. We further introduce the 
 symbols $M_\psi$, $M^\phi$, and \smash{$M_\psi^\phi$} to denote the multiplier maps 
 associated with the functions \smash{$S_\psi$, $S^\phi$, and $S_\psi^\phi$}, respectively, i.e., 
 \smash{$M_\psi, M^\phi, M_\psi^\phi\colon H^{-1}(\Omega) \to H^{-1}(\Omega)$} are defined by 
\begin{align*}
M_\psi(u) := u - (- \Delta) S_\psi(u),
~~\,
M^\phi(u) := u - (- \Delta) S^\phi (u),
~~\,
M_\psi^\phi(u) := u - (- \Delta) S_\psi^\phi (u),
\end{align*}
for all $u \in H^{-1}(\Omega)$. 
Clearly, the maps \smash{$M_\psi, M^\phi, M_\psi^\phi\colon H^{-1}(\Omega) \to H^{-1}(\Omega)$} 
are globally Lipschitz continuous too. 
From  \eqref{eq:regOmega} and classical results, we obtain
the following.

\begin{lemma}[higher regularity of multipliers and states]%
	\label{lem:regularity_again}%
	For all $u \in L^p(\Omega)$, we have
	\smash{$M_\psi(u), M^\phi(u), M_\psi^\phi (u) \in L^p(\Omega)$}
	and
	\smash{$S_\psi(u), S^\phi(u), S_\psi^\phi (u) \in C_0(\Omega)$}.
	Moreover,
	the maps \smash{$S_\psi$}, \smash{$S^\phi$}, and \smash{$S_\psi^\phi$}
	are globally Lipschitz  as functions from $L^p(\Omega)$ to $C_0(\Omega)$,
	and the maps \smash{$M_\psi$}, \smash{$M^\phi$}, and \smash{$M_\psi^\phi$}
	are globally Lipschitz as functions from $L^p(\Omega)$ to $L^1(\Omega)$.
\end{lemma}
\begin{proof}
	The assertions of the lemma follow
	from the so-called Lewy--Stampacchia inequalities; see \cite[Theorem~4.35]{Troianiello1987}.
	We include a proof for the convenience of the reader. 
	We begin with the bilateral case.
	Given a function $u \in L^p(\Omega)$ with state \smash{$y := S_\psi^\phi(u)$}, we consider
	the solution $\tilde y \in H_0^1(\Omega)$ of the variational inequality 
	\begin{equation}
		\label{eq:VI_tilde_y}
		\tilde y \ge y,
		\qquad
		\dual*{-\Delta \tilde y - \min(-\Delta\phi, u)}{\tilde v - \tilde y}_{H_0^1(\Omega)} \ge 0
		\quad \forall \tilde v \in H_0^1(\Omega), \tilde v \ge y.
	\end{equation}
	Here, the inequalities $\tilde y \ge y$ and $\tilde v \ge y$ are to be understood in the a.e.-sense.
	In order to show $\tilde y \le \phi$,
	we note that
	$y \leq \min(\tilde y, \phi) \in H_0^1(\Omega)$ holds;
	see \cref{lem:min_in_hnulleins}.
	Choosing $\min(\tilde y, \phi)$
	as the test function in \eqref{eq:VI_tilde_y}
	yields
	\begin{align*}
		0 &\le
		\int_\Omega
		\nabla\tilde y\cdot\nabla(\min(\tilde y, \phi) - \tilde y)
		- \min(-\Delta\phi, u) (\min(\tilde y, \phi) - \tilde y)
		\dd x
		\\
		&\le
		\int_\Omega
		\nabla\tilde y\cdot\nabla(\min(\tilde y, \phi) - \tilde y)
		- (-\Delta\phi) (\min(\tilde y, \phi) - \tilde y)
		\dd x
		\\
		&=
		\int_\Omega
		\nabla(\tilde y - \phi)\cdot\nabla(\min(\tilde y, \phi) - \tilde y)
		\dd x
		=
		-\norm{\min(0,\phi - \tilde y)}_{H_0^1(\Omega)}^2
		.
	\end{align*}
	Here, we have again used \cite[Theorem 5.8.2]{Attouch2006}. 
	The above shows that $\min(0,\phi - \tilde y) = 0$.
	Thus, $\psi \le y \le \tilde y \le \phi$ a.e.\ in $\Omega$.
	By choosing $\tilde y$ as the test function in the VI for $y$,
	by setting $\tilde v = y$ in 
	\eqref{eq:VI_tilde_y}, and by adding the resulting inequalities, we obtain that 
	\begin{align*}
		\norm{y - \tilde y}_{H_0^1(\Omega)}^2
		\le
		\int_\Omega (u - \min(-\Delta\phi, u))(y - \tilde y) \dd x
		\le
		0.
	\end{align*}
	This proves $y = \tilde y$.
	From \eqref{eq:VI_tilde_y}, we now get
	\begin{align*}
		-\Delta y - \min(-\Delta\phi, u)
		=
		-\Delta \tilde y - \min(-\Delta\phi, u)
		\ge
		0,
	\end{align*}
	where the inequality is to be understood in the sense of $H^{-1}(\Omega)$; cf.\ \cite[Lemma 2.5]{Rauls2020}.
	Hence, $\min(-\Delta\phi,u) \le -\Delta y$.
	Analogously, one also shows that
	$-\Delta y \le \max(-\Delta\psi,u)$.
	Thus, 
	$-\Delta y \in L^p(\Omega)$
	and, as a consequence, \smash{$M_\psi^\phi(u) = u - (-\Delta) y \in L^p(\Omega)$}.
	Due to \eqref{eq:regOmega}, this also yields 
	\smash{$S_\psi^\phi (u) \in C_0(\Omega)$}. The asserted Lipschitz 
	continuity of \smash{$S_\psi^\phi$} as a 
	function from $L^p(\Omega)$ to $C_0(\Omega)$
	now follows immediately from \cref{prop:LpLipschitz}. 
	For the unilateral cases, one easily obtains from 
	\cref{prop:LpLipschitz} that
	\smash{$S_\psi(u), S^\phi(u) \in L^\infty(\Omega)$} for 
	$u \in L^p(\Omega)$. This allows us to add artificial, inactive, constant 
	obstacles to the VIs satisfied by $S_\psi(u)$ and $S^\phi(u)$ and to deduce
	the regularities \smash{$M_\psi(u), M^\phi(u) \in L^p(\Omega)$}
	and
	\smash{$S_\psi(u), S^\phi(u) \in C_0(\Omega)$} from the already established
	result for the bilateral case. The Lipschitz continuity 
	of $S_\psi$ and $S^\phi$ then again follows from 
	\cref{prop:LpLipschitz}. To finally establish the Lipschitz continuity of 
	\smash{$M_\psi$}, \smash{$M^\phi$}, \smash{and $M_\psi^\phi$} from $L^p(\Omega)$
	to $L^1(\Omega)$, 
	one can argue along the exact same lines as in 
	\cite[proof of Theorem 2.3]{ChristofWachsmuth2021SSC}. 
\end{proof}

To be able to define suitable Newton derivatives for 
$S_\psi$ and $S^\phi$, we introduce some additional notation.

\begin{definition}[solution maps of Poisson problems]%
\label{def:PoissonSolutionOps}%
	Given an open nonempty set $O \subset \Omega$, we denote by
	$\SSS(O) \colon H^{-1}(\Omega) \to H_0^1(\Omega)$
	the solution map $h \mapsto \delta$ of the Poisson problem
\begin{align*}
			\delta \in H_0^1(O), \qquad
			\dual{ - \Delta \delta - h}{ v }_{H_0^1(\Omega)} = 0
			\qquad \forall v \in H_0^1(O).
\end{align*}
		Here, $H_0^1(O) \subset H_0^1(\Omega)$
		is the completion of $C_c^\infty(O)$ in $H^1(\Omega)$; cf.\ \cref{sec:2}.
\end{definition}

Using \cref{def:PoissonSolutionOps}, we can define the set-valued functions 
that serve as the Newton derivatives for 
$S_\psi$, $S^\phi$, $M_\psi$, and $M^\phi$ in this section.

\begin{definition}[{Newton derivatives for $S_\psi$, $S^\phi$, $M_\psi$, and $M^\phi$}]%
\label{def:GenDerUni}%
 Given $u \in L^p(\Omega)$, we define:
\begin{equation*}
\begin{aligned}
\D S_\psi(u) &:= 
\big \{
\SSS(O)
\,\left |~
O \text{ open and }
\set{S_\psi(u) > \psi } \subset O \subset \Omega \setminus \supp\left (M_\psi(u)\right )
\big \},\right.
\\
\D S^\phi(u) &:= 
\big \{
\SSS(O)
\,\left |~
O \text{ open and }
\set{S^\phi(u) < \phi} \subset O \subset \Omega \setminus \supp\left (M^\phi(u)\right )
\big \},\right.
\\
\D M_\psi(u) &:= \Id - (-\Delta) \D S_\psi(u),
\\
\D M^\phi (u) &:= \Id - (-\Delta) \D S^\phi(u).
\end{aligned}
\end{equation*}
Here, the sets $\set{S_\psi(u) > \psi }$ and $\set{S^\phi(u) < \phi}$ are defined 
with respect to (w.r.t.)
the $C(\Omega)$-representatives of the functions $S_\psi(u)$, $S^\phi(u)$, $\psi$, and $\phi$,
respectively,
and the symbol $\supp(\cdot)$ denotes the support in the sense of Borel measures on $\Omega$.
\end{definition}

We note that the sets in \cref{def:GenDerUni}
are nonempty
since we can always choose
$O = \set{S_\psi(u) > \psi}$
and
$O = \set{S^\phi(u) < \phi}$,
respectively.
From the results of \cite{Rauls2020}, 
we obtain the following.

\begin{lemma}%
\label{lem:generalderivativeUni}%
Let $u \in L^p(\Omega)$.\;Then 
$\D S_\psi(u) \subset \partial_B^{sw}S_\psi(u)$ and 
$\D S^\phi(u) \subset \partial_B^{sw}S^\phi(u)$.
\end{lemma}

\begin{proof}
This follows from \cite[Theorem 4.3]{Rauls2020},
the fact that open sets are quasi-open, 
the fact that the $H_0^1(\Omega)$-fine-support of a measure 
is contained in its ordinary support (see \cite[Lemma A.4]{Wachmuth2014}),
and the inclusions in \cite[Proposition 2.11]{Rauls2020}. Note that the 
result for $S^\phi$ can be trivially obtained from that for $S_\psi$ by a simple
switch of signs so that the results of \cite{Rauls2020} indeed also apply to $S^\phi$.
\end{proof}

In combination with \cref{th:semismoothH01unilateral},
\cref{lem:generalderivativeUni} shows that the maps
$\D S_\psi, \D S^\phi\colon$ $L^p(\Omega) \rightrightarrows\LL(H^{-1}(\Omega), H_0^1(\Omega))$
and
$\D M_\psi, \D M^\phi\colon L^p(\Omega) \rightrightarrows\LL(H^{-1}(\Omega), H^{-1}(\Omega))$
can be used as Newton derivatives for 
$S_\psi$, $S^\phi$, $M_\psi$, and $M^\phi$,
respectively.

\begin{corollary}%
The operators
$S_\psi$ and  $S^\phi$
are Newton differentiable as functions from $L^p(\Omega)$ to $H_0^1(\Omega)$ 
when equipped with the derivatives 
$\D S_\psi$ and $\D S^\phi$, respectively,
and the operators $M_\psi$ and $M^\phi$
are Newton differentiable as functions from $L^p(\Omega)$ to $H^{-1}(\Omega)$ 
when equipped with the derivatives 
$\D M_\psi$ and $\D M^\phi$, respectively.
\end{corollary}

\begin{proof}
For $S_\psi$, the assertion follows from 
\cref{th:semismoothH01unilateral}, 
\cref{lem:generalderivativeUni},
the estimate $p > \max(1,d/2) \ge \max(1, 2d/(d+2))$,
the fact that $\D S_\psi(u) $ is nonempty for all $u \in L^p(\Omega)$
(cf.\ \cref{def:GenDerUni}),
and the definition of Newton differentiability. 
The result for $S^\phi$ is obtained from that for $S_\psi$
by an elementary sign transformation. 
As the maps $M_\psi$ and $M^\phi$ arise from $S_\psi$ and $S^\phi$,
respectively, by applying the Laplacian and by adding the identity, 
the assertions for $M_\psi$ and $M^\phi$ follow 
from the results for  $S_\psi$ and $S^\phi$ and the definition of Newton differentiability;
cf.\ \cref{def:GenDerUni}.~\end{proof}

\begin{remark}~
\begin{itemize}
\item As the reader may have noticed, we do not endow the maps 
$S_\psi$ and  $S^\phi$ with their full strong-weak Bouligand differentials 
$\partial_B^{sw}S_\psi $ and 
$\partial_B^{sw}S^\phi$ as Newton derivatives in this section although 
we know from \cref{th:semismoothH01unilateral} that we could do so and still 
obtain Newton differentiability. The reason for this is that we require 
additional properties of the Newton derivatives of $S_\psi$ and  $S^\phi$
in the following
to be able to invoke the results of \cref{sec:5}; cf.\ \cref{lem:Tinversion}.

\item The subsequent analysis still works when 
the concepts of openness, inclusion, and support in \cref{def:GenDerUni}
are replaced by their capacitary counterparts; cf.\ \cite[Theorem 4.3]{Rauls2020}.
We omit this generalization here for the sake of 
simplicity. 

\end{itemize}
\end{remark}

We can now start proving the Newton differentiability of $S_\psi^\phi$ and $M_\psi^\phi$. 
We begin with the following observation.

\begin{proposition}%
\label{prop:impsys}%
Suppose that $u \in L^p(\Omega)$
and  $y \in H_0^1(\Omega) $ are given. 
Then it holds $y = S_\psi^\phi(u) $ if and only if there 
exist $ \lambda_\psi,\lambda^\phi \in L^p(\Omega)$ satisfying 
\begin{equation}
\label{eq:impsys1d}
\begin{gathered}
-\Delta y = u - \lambda_\psi - \lambda^\phi, 
\qquad
\lambda_\psi = M_\psi ( u - \lambda^\phi ),
\qquad
\lambda^\phi = M^\phi ( u - \lambda_\psi ).
\end{gathered}
\end{equation}
If $y = S_\psi^\phi(u) $ holds, then the functions
$\lambda_\psi$ and $\lambda^\phi$ are 
unique and one has
\[
\lambda_\psi  = \min(0, M_\psi^\phi(u)),
\quad
\lambda^\phi  = \max(0, M_\psi^\phi(u)),
\quad
\text{and}\quad  y =  S_\psi( u - \lambda^\phi) = S^\phi(u - \lambda_\psi).
\]
\end{proposition}

\begin{proof}
Suppose that 
$y = S_\psi^\phi(u) $  holds. Then $-\Delta y = u - M_\psi^\phi(u)$ is in $L^p(\Omega)$
by \cref{lem:regularity_again} and the VI for $y$ reads
\begin{equation}
\label{eq:randomeq2727}
\int_\Omega 
(-\Delta y   - u)(v - y) \dd x \geq 0\qquad \forall v \in H_0^1(\Omega) \text{ with } \psi \leq v \leq \phi \text{ a.e.\ in }\Omega.
\end{equation}
By exploiting our assumptions on $\psi$ and $\phi$ and the continuity of $y$,
and by choosing functions of the form $v = y + z$, $z \in C_c^\infty(\Omega)$,
in \eqref{eq:randomeq2727},
one easily checks that
\begin{equation}
\label{eq:randomeq3737hdd8e-1d}
- M_\psi^\phi(u)
=
-\Delta y - u 
\begin{cases}
= 0 & \text{ a.e.\ in } \{\psi < y < \phi\},
\\
\geq 0 & \text{ a.e.\ in } \{y = \psi\},
\\
\leq 0 & \text{ a.e.\ in } \{y =\phi\},
\end{cases}
\end{equation}
where the sets on the right are defined w.r.t.\ the $C( \Omega)$-representatives of 
$y$, $\psi$, and $\phi$.
If we set
$\lambda_\psi := \mathds{1}_{\{y = \psi\} } (u - (-\Delta) y ) \in L^p(\Omega)$
and
$\lambda^\phi := \mathds{1}_{\{y = \phi\} } (u - (-\Delta) y ) \in L^p(\Omega)$,
then \eqref{eq:randomeq3737hdd8e-1d} yields that
\begin{equation}
\label{eq:randomsys26365465-1d}
\begin{gathered}
-\Delta y = u - \lambda_\psi - \lambda^\phi,
\\
\lambda_\psi \leq 0 \text{ a.e.\ in }\Omega,
\qquad 
y - \psi \geq 0 \text{ a.e.\ in } \Omega,
\qquad
\lambda_\psi(y - \psi) = 0 \text{ a.e.\ in } \Omega,
\\
\lambda^\phi \geq 0 \text{ a.e.\ in }\Omega,
\qquad 
y - \phi \leq 0 \text{ a.e.\ in } \Omega,
\qquad
\lambda^\phi(y - \phi) = 0 \text{ a.e.\ in } \Omega.
\end{gathered}
\end{equation}
This implies
$y = S_\psi( u - \lambda^\phi)$,
$y = S^\phi(u - \lambda_\psi)$,
$\lambda_\psi = M_\psi(u - \lambda^\phi)$,
and 
$\lambda^\phi = M^\phi(u - \lambda_\psi)$
and proves the implication ``$\Rightarrow$''
as well as the last sentence of the assertion.\pagebreak

To see that  ``$\Leftarrow$'' also holds, 
let us assume that 
$u, \lambda_\psi, \lambda^\phi \in L^p(\Omega)$ and $y \in H_0^1(\Omega)$
satisfying \eqref{eq:impsys1d} are given. Then
the identities $\lambda_\psi = M_\psi  ( u - \lambda^\phi)$,
$\lambda^\phi = M^\phi ( u - \lambda_\psi )$,
and 
$-\Delta y = u - \lambda_\psi - \lambda^\phi$ imply that 
$y =  S_\psi( u - \lambda^\phi) = S^\phi(u - \lambda_\psi)$ holds. 
By the definitions of $S_\psi$, $M_\psi$, $S^\phi$,
and $M^\phi$, this entails that 
$y$, $\lambda_\psi$, $u$, and $\lambda^\phi$ satisfy \eqref{eq:randomsys26365465-1d}.
This complementarity system implies \eqref{eq:randomeq3737hdd8e-1d}
which yields \smash{$y = S_\psi^\phi(u) $} as desired.~ 
\end{proof}

Consider now the Hilbert space
\begin{align*}
\HH :=
H^{-1}(\Omega)
\times
H^{-1}(\Omega)
\end{align*}
(endowed with its canonically induced norm)
and the functions defined by
\begin{equation*}
\begin{gathered}
F\colon\HH \times L^p(\Omega)
\to
\HH,
\qquad
\begin{pmatrix}
\lambda_\psi
\\
\lambda^\phi
\\
u
\end{pmatrix}
\mapsto
\begin{pmatrix}
\lambda_\psi
-
M_\psi(u - \lambda^\phi)
\\
\lambda^\phi
-
M^\phi(u - \lambda_\psi)
\end{pmatrix},
\end{gathered}
\end{equation*}
and
\begin{equation}
\label{eq:PhiDef}
\Phi\colon L^p(\Omega)
\to L^p(\Omega)^2,
\qquad
u \mapsto
\begin{pmatrix}
\min(0, M_\psi^\phi(u))
\\
\max(0, M_\psi^\phi(u))
\end{pmatrix}
=:
\begin{pmatrix}
\Phi_\psi(u)
\\
\Phi^\phi(u)
\end{pmatrix}
.
\end{equation}
Then we have the following.

\begin{lemma}%
\label{lem:1dPhiCont}%
It holds $F(\Phi(u), u) = 0 \in \HH$ for all $u \in L^p(\Omega)$.
\end{lemma}

\begin{proof}
See \eqref{eq:PhiDef} and \cref{prop:impsys}.
\end{proof}

In what follows, our plan is to 
argue as in the proof
of the implicit function theorem to establish the Newton differentiability of 
$\Phi$ as a function from $L^p(\Omega)$ to $\HH$.
To pursue this approach, we require a generalized derivative for $F$.

\begin{definition}[Newton derivative of $F$]%
\label{def:def.4.12}%
We denote by
$\D F$ the set-valued map
from $L^p(\Omega)^3$
to  
$\LL(\HH, \HH) \times \LL(L^p(\Omega), \HH)$
given by
\begin{align*}
	\D F(\lambda_\psi, \lambda^\phi, u) 
	:=
	\set*{ \GG =
		\left (
		 \begin{pmatrix}
	\Id   & G_\psi 
		\\
	G^\phi  	& \Id
	\end{pmatrix},	
	\begin{pmatrix}
		- G_\psi 
		\\
		- G^\phi 
	\end{pmatrix}
	\right )
	\given
	 \begin{aligned}
	   G_\psi &\in \D M_\psi(u - \lambda^\phi),\\
	   G^\phi &\in \D M^\phi(u - \lambda_\psi) 
	 \end{aligned}
	}.
\end{align*}
The two components of $\GG$ will be denoted by
$\GG_1 \in \LL(\HH,\HH)$
and
$\GG_2 \in \LL(L^p(\Omega),\HH)$.
\end{definition}

Now,
one is tempted to use a standard estimate
to establish that the implicit function 
$\Phi$
is Newton differentiable from $L^p(\Omega)$
to  $\HH$ with Newton derivative 
$\D\Phi(u) := \{ - \GG_1^{-1}\GG_2 \mid \GG \in \D F(\Phi(u),u)\}$.
Given $u \in L^p(\Omega)$, $h \in L^p(\Omega) \setminus \{0\}$, and 
$G := - \GG_1^{-1}\GG_2 \in \D\Phi(u + h)$ with
$\GG \in \D F(\Phi(u+h), u+h)$,
one would try
\begin{equation}
\label{eq:longequation22}
\begin{aligned}
	&
	\frac{\norm{ \Phi(u + h) - \Phi(u) - G h }_\HH}{\norm{h}_{L^p(\Omega)}}
	=
	\frac{\norm{ \Phi(u + h) - \Phi(u) +\GG_1^{-1} \GG_2 h }_\HH}{\norm{h}_{L^p(\Omega)}}
	\\&
	\leq
	C
	\frac{\norm{ \GG_1 \left (\Phi(u + h) - \Phi(u) \right )+ \GG_2 h }_\HH}{\norm{h}_{L^p(\Omega)}}
	=
	C
	 \frac{	 \left \|
	  (\GG_1~~\GG_2)
	 	\begin{pmatrix}
		\Phi(u + h) - \Phi(u) 
		\\
		h
	\end{pmatrix}\right \|_\HH}{\norm{h}_{L^p(\Omega)}}
	\\&
	=
	C
		 \frac{ \left \|
	 F(\Phi(u + h), u+h) - F(\Phi(u), u) - 
	 \GG
	 	\begin{pmatrix}
		\Phi(u + h) - \Phi(u) 
		\\
		h
	\end{pmatrix}\right \|_\HH}{\norm{h}_{L^p(\Omega)}}
	\\
	&\leq
	C L_\XX
	 \frac{\left \|
	 F(\Phi(u + h), u+h) - F(\Phi(u), u) - 
	 \GG
	 	\begin{pmatrix}
		\Phi(u + h) - \Phi(u) 
		\\
		h
	\end{pmatrix}\right \|_\HH}{ \left \|
	 	\begin{pmatrix}
		\Phi(u + h) - \Phi(u) 
		\\
		h
	\end{pmatrix}\right \|_{\XX}}
	.
\end{aligned}
\end{equation}
Here, $C>0$ denotes a bound on the $\LL(\HH, \HH)$-norm 
of the inverse $\GG_1^{-1}$ that is
(at least for all small enough $h$) independent of $h$ and $\GG_1$ (and whose existence, 
of course, has to be proved);
$\XX$ is a suitable normed space;
and $L_\XX$ is a (local) Lipschitz constant 
of the function $(\Phi, \Id)^\top\colon L^p(\Omega) \to \XX$; cf.\ \cite{Mannel2018}
and the references therein. 
Ideally, the last term in \eqref{eq:longequation22} would converge to zero for
$0 <  \|h\|_{L^p(\Omega)} \to 0$.

Unfortunately, these standard arguments fail 
for all $d>1$
as they require a space $\XX$ 
such that $F$ is Newton differentiable as a map from $\XX$ to $\HH$
and such that $(\Phi, \Id)^\top$ is locally Lipschitz as a function from $L^p(\Omega)$ to $\XX$.
To the best of our knowledge,
such a space $\XX$ is not available for $d>1$  
as the maps $M_\psi$ and $M^\phi$
are Newton differentiable only
from $L^q(\Omega)$ to $H^{-1}(\Omega)$
and since the map \smash{$M_\psi^\phi$} is
Lipschitz only from
$L^q(\Omega)$ to $H^{-1}(\Omega) \cap L^1(\Omega)$
for all $q > \max(1, 2d/(d+2))$ and $d > 1$. (For $d=1$, the arguments
in \eqref{eq:longequation22} can be shown to work with $\XX = L^1(\Omega)^3$.)
We will therefore argue differently and 
exploit the structure of $F$ and $\Phi$
to show directly that
\begin{equation}
	\label{eq:weird_fraction}
\hspace{-0.15cm}\sup_{\GG \in \D F(\Phi(u+h),u+h)} \hspace{-0.15cm}
	\frac{\norm{ \GG_1 (\Phi(u + h) - \Phi(u) )+ \GG_2 h }_\HH}{\norm{h}_{L^p(\Omega)}}
	\to 0~~\text{for } 0 <  \|h\|_{L^p(\Omega)} \to 0.
\end{equation}

The first main step of our analysis is proving  that, 
for every $u \in L^p(\Omega)$,
there indeed exist constants $r > 0$ and $C>0$ satisfying
\begin{equation}
\label{eq:randomeq267363}
\sup_{\GG \in \D F(\Phi(u+h),u+h)} 
\|\GG_1^{-1}\|_{\LL(\HH, \HH)}
\leq
C
\qquad \forall
h \in L^p(\Omega),
\|h\|_{L^p(\Omega)} \leq r.
\end{equation}
To establish this locally uniform invertibility,
we verify that the operators $G_\psi$ and $G^\phi$
appearing in the definition of $\D F(\Phi(u ),u )$
are orthogonal projections in $H^{-1}(\Omega)$.
\begin{lemma}%
\label{lem:projidentification}%
Let $u \in L^p(\Omega)$.
Let 
$G_\psi = \Id - (-\Delta)\SSS(O_\psi)\in \D M_\psi(u - \Phi^\phi(u))$ 
and 
$G^\phi= \Id - (-\Delta)\SSS(O^\phi)\in \D M^\phi (u - \Phi_\psi(u))$ be given,
where $\Phi_\psi(u), \Phi^\phi(u) \in L^p(\Omega)$ denote the components of $\Phi(u)$
and $O_\psi$ and $O^\phi$  the open sets that generate $G_\psi$ and $G^\phi$, respectively,
according to \cref{def:GenDerUni}.
 Define
\begin{equation}
\label{eq:Wdefs}
\begin{aligned}
W(G_\psi) &:= \bracks*{- \Delta H_0^1(O_\psi) }^\perp
\qquad
\text{and}
\qquad
&
W(G^\phi)  &:=  \bracks*{- \Delta H_0^1(O^\phi) }^\perp,
\end{aligned}
\end{equation}
where the orthogonal complement is taken in $H^{-1}(\Omega)$,
and denote 
the \smash{$(\cdot,\cdot)_{H^{-1}(\Omega)}$}-orthogonal projection
onto a closed subspace $W \subset H^{-1}(\Omega)$ by $P_W$.
 Then it holds 
\begin{align*}
G_\psi
=
P_{ W(G_\psi)}
\qquad
\text{and}
\qquad
G^\phi
=
P_{ W(G^\phi)}.
\end{align*}

\end{lemma}

\begin{proof}
From \cref{def:PoissonSolutionOps}, it follows that 
$\SSS(O_\psi)(-\Delta)$ is the $H_0^1(\Omega)$-orthogonal 
projection onto $H_0^1(O_\psi)$. In view of  \eqref{eq:H-1scalarproduct},
this implies that $(-\Delta)\SSS(O_\psi)$ is the
$H^{-1}(\Omega)$-orthogonal projection onto $(-\Delta)H_0^1(O_\psi)$. 
Thus, $G_\psi = \Id - (-\Delta)\SSS(O_\psi)$ is the $H^{-1}(\Omega)$-orthogonal projection 
onto $W(G_\psi) = [ - \Delta H_0^1(O_\psi)]^\perp$ as claimed. 
The assertion for 
$G^\phi = \Id - (-\Delta)\SSS(O^\phi)$ is proved analogously. 
\end{proof}

Next, we are going to show that
the angle between the spaces
$W(G_\psi)$
and  
$W(G^\phi)$ in \eqref{eq:Wdefs}
 is positive.
To this end, we need a lemma on the stability of contact sets.
\begin{lemma}%
\label{lem:sepcontact}%
Let $u \in L^p(\Omega)$ be given. Then there exist
a constant
$r > 0$
and nonempty disjoint compact sets $Z_\psi, Z^\phi \subset \Omega$ such that
it holds 
\begin{align*}
\{S_\psi^\phi(u + h) = \psi\} \subset Z_\psi
\qquad
\text{and}
\qquad
\{S_\psi^\phi(u + h) = \phi\} \subset Z^\phi
\end{align*}
for all $h \in L^p(\Omega)$ with $\|h\|_{L^p(\Omega)} \leq r$.
Here, the contact sets are again defined w.r.t.\ the $C(\Omega)$-representatives
of the functions \smash{$S_\psi^\phi(u + h)$, $\phi$, and $\psi$}.
\end{lemma}
\begin{proof}
If \smash{$\{ S_\psi^\phi(u)  = \psi \} = \{ S_\psi^\phi(u) = \phi\} = \emptyset$} holds,
then the continuity of \smash{$S_\psi^\phi(u)$}, $\phi$, and $\psi$ implies that there exists $m > 0$ with 
\smash{$\psi  + m \leq S_\psi^\phi(u) \leq \phi - m$ in $\bar\Omega$}
and it follows from the Lipschitz continuity of 
$S_\psi^\phi$ as a map from $L^p(\Omega)$
to $C_0(\Omega)$ in \cref{lem:regularity_again}
that the contact sets of
\smash{$S_\psi^\phi(u + h)$} are empty for all small enough $\|h\|_{L^p(\Omega)}$. 
In this case, the assertion holds with any disjoint nonempty compact $Z_\psi, Z^\phi \subset \Omega$. 

Let us now suppose that 
\smash{$\{S_\psi^\phi(u) = \psi\} \neq\emptyset$}
and let $c>0$ be the constant from \eqref{eq:constant_c}.
Then our assumptions on $\psi$, the continuity of \smash{$S_\psi^\phi(u)$},
and the homogeneous Dirichlet boundary conditions
imply that
$Z_\psi := \{ x \in \Omega \mid S_\psi^\phi(u)(x) \leq  \psi(x) + c/4 \}$
is a nonempty and compact set. 
Again due to the Lipschitz continuity of 
\smash{$S_\psi^\phi$} as a function from $L^p(\Omega)$
to $C_0(\Omega)$, 
we obtain that we can find 
$r > 0$ such that, for all $h \in L^p(\Omega)$
with $\|h\|_{L^p(\Omega)} \leq r$, 
we have 
\smash{$S_\psi^\phi(u + h)(x) >  \psi(x) + c/8$}
for all $x \in \Omega \setminus Z_\psi$. 
This implies that 
\smash{$\{S_\psi^\phi(u + h) = \psi\}  \subset Z_\psi$}
holds for all such $h$ as desired.
The set $Z^\phi$ is constructed analogously
(or, in the case \smash{$\{S^\phi_\psi(u) = \phi\} = \emptyset$},  chosen 
as an arbitrary nonempty compact subset of $\Omega \setminus Z_\psi \neq \emptyset$
after potentially making $r$ smaller).
\end{proof}

Via \cref{lem:sine_angle}, we can now show that 
the spaces $W(G_\psi)$
and  
$W(G^\phi)$ in \eqref{eq:Wdefs}
indeed enclose a positive angle that is stable w.r.t.\ small perturbations of $u$.

\begin{lemma}%
\label{lem:boundproj}%
For every $u \in L^p(\Omega)$, there exist constants $r > 0$, $\bar s_0 \in (0,1]$ 
with
\begin{equation}
\label{eq:anglebound}
\inf_{
\substack{
h \in L^p(\Omega), 
\\
\|h\|_{L^p(\Omega)} \leq r}
}
\;
\inf_{
\substack{
G_\psi \in D M_\psi(u +h  - \Phi^\phi(u + h)),
\\ 
G^\phi \in D M^\phi(u +h  - \Phi_\psi(u + h))}
}
s_0(W(G_\psi), W(G^\phi))
\geq  \bar s_0.
\end{equation}
Here, $W(G_\psi)$
and  
$W(G^\phi)$ are defined in \eqref{eq:Wdefs}
and $s_0$ is defined in \cref{lem:sine_angle}.
\end{lemma}
\begin{proof}
Let $u \in L^p(\Omega)$ be fixed and let $r > 0$, $Z_\psi$, and $Z^\phi$ be
chosen for $u$
as in  \cref{lem:sepcontact}.
Let $\zeta \in C_c^\infty(\Omega)$ be a function that satisfies 
$\zeta \equiv 1$ on $Z_\psi$ and $\zeta \equiv 0$ on $Z^\phi$. We define 
$J\colon H^{-1}(\Omega) \to H^{-1}(\Omega)$ to be the operator that 
maps a functional $v \in H^{-1}(\Omega)$ to the element $J(v)$ of $H^{-1}(\Omega)$ defined by 
\begin{equation}
\label{eq:Jdef}
\left \langle J(v), z\right \rangle_{H_0^1(\Omega)}
:=
\left \langle v, \zeta z  \right \rangle_{H_0^1(\Omega)}\qquad \forall z \in H_0^1(\Omega).
\end{equation}
Note that
$J \in \LL(H^{-1}(\Omega), H^{-1}(\Omega)) \setminus \{0\}$. 
Consider now some $h \in  L^p(\Omega)$ 
satisfying 
$\| h\|_{L^p(\Omega)} \leq r$
and let
$G_\psi \in D M_\psi(u +h  - \Phi^\phi(u + h))$
and 
$G^\phi \in D M^\phi(u +h  - \Phi_\psi(u + h))$ be given.
Denote the open sets 
that generate $G_\psi$ and $G^\phi$ according to \cref{def:GenDerUni}
by $O_\psi$ and $O^\phi$, respectively.
Then it 
follows from
the \mbox{inclusions}
\smash{$\{S_\psi^\phi(u + h) = \psi\} \subset Z_\psi$,}
 $\{S_\psi^\phi(u + h) = \phi\} \subset Z^\phi$,
 $\{S_\psi^\phi(u + h) > \psi\} \subset O_\psi$,
and 
 $\{S_\psi^\phi(u + h) < \phi\} \subset O^\phi$
(see \cref{def:GenDerUni}, \cref{prop:impsys}, \eqref{eq:PhiDef}, and \cref{lem:sepcontact});
the properties of $\zeta$;
and the definitions in \eqref{eq:Wdefs}
that, 
for all 
 $v \in W(G_\psi)$,
$w \in W(G^\phi)$, and $z \in H_0^1(\Omega)$,
we have 
\begin{equation*}
\begin{aligned}
\left \langle J(v + w), z\right \rangle_{H_0^1(\Omega)}
&=
\left \langle v + w, \zeta z  \right \rangle_{H_0^1(\Omega)}
&&\text{by \eqref{eq:Jdef}}
\\
&=
\left ( v + w, - \Delta (\zeta z )  \right )_{H^{-1}(\Omega)}
&&\text{by \eqref{eq:H-1scalarproduct}}
\\
&=
\left ( v, - \Delta (\zeta z ) \right )_{H^{-1}(\Omega)}
&&\text{by \eqref{eq:Wdefs} and $\zeta z \in H_0^1(O^\phi)$}
\\
&=
\left ( v, - \Delta z  \right )_{H^{-1}(\Omega)}
&&\text{by \eqref{eq:Wdefs} and $(1 - \zeta) z \in H_0^1(O_\psi)$}
\\
&=
\left \langle v, z \right   \rangle_{H_0^1(\Omega)}
&&\text{by \eqref{eq:H-1scalarproduct}.}
\end{aligned}
\end{equation*}
Thus, $J(v + w) = v$ and, as a consequence,
\begin{align*}
\|J\|_{\LL(H^{-1}(\Omega), H^{-1}(\Omega))}
\|v + w\|_{H^{-1}(\Omega)}
\geq
\|v\|_{H^{-1}(\Omega)}
\quad
\forall 
 v \in W(G_\psi) ,
 w \in W(G^\phi).
\end{align*}
Due to the definition of $s_0(W(G_\psi), W(G^\phi))$ in \cref{lem:sine_angle}, this 
yields 
\begin{align*}
s_0(W(G_\psi), W(G^\phi)) \geq 
\min\big (1, \|J\|_{\LL(H^{-1}(\Omega), H^{-1}(\Omega))}^{-1}\big) > 0.
\end{align*}
As $h$, $G_\psi$, and $G^\phi$ were arbitrary, the estimate
\eqref{eq:anglebound} now follows immediately 
(with $\bar s_0 := \min(1, \|J\|_{\LL(H^{-1}(\Omega), H^{-1}(\Omega))}^{-1}) > 0$).
This completes the proof.
\end{proof}

As a consequence of 
\cref{thm:some_invertibility,lem:boundproj,lem:projidentification},
we now obtain the locally uniform invertibility in \eqref{eq:randomeq267363}.

\begin{lemma}%
\label{lem:Tinversion}%
For every $u \in L^p(\Omega)$,
there exist constants $r, C> 0$ such that,
for all 
$h \in L^p(\Omega)$ with
$\|h\|_{L^p(\Omega)} \leq r$
and for all $\GG \in \D F(\Phi(u+h),u+h)$,
the operator $\GG_1 \in \LL(\HH, \HH)$
is continuously invertible
and
\begin{equation}
\label{eq:D1_uniform_inverse}
\sup_{\GG \in \D F(\Phi(u+h),u+h)} 
\|\GG_1^{-1}\|_{\LL(\HH, \HH)}
\leq
C
\qquad \forall
h \in L^p(\Omega),
\|h\|_{L^p(\Omega)} \leq r.
\end{equation}
\end{lemma}

\begin{proof}
From 
\cref{def:def.4.12,lem:projidentification}, we obtain that the 
operators $\GG_1$ appearing in $\D F(\Phi(u+h),u+h)$
have precisely the form \eqref{eq:R1def}.
From \cref{lem:boundproj}, it follows further that 
we can find a constant $r > 0$ such that 
the subspaces of $H^{-1}(\Omega)$ appearing in the 
formulas for the operators $\GG_1$
enclose an angle whose sine is bounded from below by 
a number $\bar s_0 \in (0,1]$ for all $h \in L^p(\Omega)$, 
$\|h\|_{L^p(\Omega)} \leq r$. 
By invoking \cref{thm:some_invertibility} and \eqref{eq:pythagoras}, 
the assertion of the lemma now follows immediately with
$C:= 4(1 - (1 - \bar s_0^2)^{1/2})^{-1}$. 
This completes the proof. 
\end{proof}

Having established \eqref{eq:D1_uniform_inverse}, we can turn our attention to the proof of \eqref{eq:weird_fraction}. We begin with an auxiliary result on  the regularity properties 
of the solution of the Poisson problem away from the support of the right-hand side.
\begin{lemma}%
	\label{lem:regularity}%
	Let $Z \subset \Omega$ be nonempty  and 
	compact and let $\zeta \in C_c^\infty(\Omega)$
	be given such that
	$\zeta \equiv 0$ holds on a neighborhood of $Z $.
	Identify $L^p(Z )$ with a subset of $L^p(\Omega) \subset H^{-1}(\Omega)$
	via the canonical embedding $L^p(Z ) \hookrightarrow L^p(\Omega)$.
	Then there exists a constant $C > 0$ such that
	\begin{align*}
		\norm{ \zeta (-\Delta)^{-1} f }_{W^{2,p}(\Omega)}
		\le
		C \norm{ f }_{L^1(Z)}\qquad \forall f \in L^p(Z ).
	\end{align*}
\end{lemma}
\begin{proof}
We assume that $\zeta \not \equiv 0$. (For $\zeta \equiv 0$, the assertion is trivial.)
Let $f \in L^p( Z)$ be given
and define  $y := (-\Delta)^{-1} f \in H_0^1(\Omega)$.
	By a duality argument,
	we get
	\begin{equation}
	\label{eq:randomeq2828nh}
	\begin{aligned}
		\norm{y}_{L^1(\Omega)}
		&=
		\sup_{\norm{z}_{L^\infty(\Omega)} \le 1} \int_\Omega y z \dd x
		=
		\sup_{\norm{z}_{L^\infty(\Omega)} \le 1} \int_\Omega f (-\Delta)^{-1} z \dd x
		\\
		&\le
		\sup_{\norm{z}_{L^\infty(\Omega)} \le 1} \norm{f}_{L^1(\Omega)} \norm{(-\Delta)^{-1} z}_{L^\infty(\Omega)}
		\le
		C \norm{f}_{L^1(\Omega)},
	\end{aligned}
	\end{equation}
	where, in the last step, we used
	the estimate
	\smash{$\norm{(-\Delta)^{-1} z}_{L^\infty(\Omega)} \le C \norm{z}_{L^\infty(\Omega)}$}
	that is a special case of \cref{prop:LpLipschitz}.
	As the support $\supp(\zeta)$ of $\zeta$ and
	the set $Z$ are nonempty, 
	disjoint, and compact,  we can further find  $\varepsilon > 0$
	such that the set
	\begin{align*}
		U 
		:=
		\set{
			x \in \Omega
			\given
			\dist(x, \supp(\zeta)) < \varepsilon
		}
	\end{align*}
	does not intersect $\partial\Omega \cup Z$.
	Now, the solution $y = (-\Delta)^{-1} f \in H_0^1(\Omega)$ satisfies
	\begin{align*}
	      \innerp{y}{-\Delta v}_{L^2(\Omega)}
		=
	      \innerp{y}{v}_{H_0^1(\Omega)}
		=
		\innerp{f}{v}_{L^2(\Omega)}
		=
		0
		\qquad
		\forall
		v \in C_c^\infty(U )
		.
	\end{align*}
	By Weyl's lemma, see \cite[Proposition~2.14]{Ponce2016},
	this implies that $y$ is harmonic on $U $.
	From the mean value property of harmonic maps, it now follows
	that  the convolution $\eta \mathbin{\star} y$ 
	of $y$ with a rotationally symmetric mollifier $\eta$ of radius smaller $\varepsilon$
	satisfies $y = \eta \mathbin{\star} y$ on $\supp(\zeta)$ and that
	the partial derivatives of $y$  satisfy 
	 $\partial^\alpha y = (\partial^\alpha \eta) \mathbin{\star} y$ on $\supp(\zeta)$
	for all $\alpha \in \N_0^d$.
	Owing to Young's inequality for convolutions, this gives
	\begin{align*}
		\norm{\partial^\alpha y}_{L^\infty(\supp(\zeta))}
		\le
		\norm{y}_{L^1(U )} \norm{\partial^\alpha \eta}_{L^\infty(\R^d)}
		=
		C_\alpha \norm{y}_{L^1(U )}
		.
	\end{align*}
	Together with the product rule and \eqref{eq:randomeq2828nh},
	this finishes the proof.~
\end{proof}

Next, we prove a result that allows us to bridge the regularity gap in \eqref{eq:longequation22}.

\begin{lemma}%
\label{lem:reg_bootstrap}%
Let $u \in L^p(\Omega)$ be given
and let $r$, $Z_\psi$, and $Z^\phi$ be as in \cref{lem:sepcontact}.
Then there exist
$\zeta_\psi, \zeta^\phi \in C_c^\infty(\Omega)$
satisfying 
$0 \leq \zeta_\psi, \zeta^\phi \leq 1$ in $\Omega$,
$\zeta_\psi \equiv 1 - \zeta^\phi \equiv 0$ on a neighborhood of $Z_\psi$,
and 
$\zeta^\phi \equiv 1 - \zeta_\psi \equiv 0$ on a neighborhood of $Z^\phi$ 
such that 
\begin{equation}
\label{eq:Midentities}
\begin{aligned}
M_\psi(u + h - \Phi^\phi(u + h))
&=
M_\psi(u + h - (-\Delta)\zeta^\phi (-\Delta)^{-1}\Phi^\phi(u + h)),
\\
M^\phi(u + h - \Phi_\psi(u + h))
&=
M^\phi(u + h - (-\Delta)\zeta_\psi (-\Delta)^{-1}\Phi_\psi(u + h)),
\\
\D M_\psi(u + h - \Phi^\phi(u + h))
&=
\D M_\psi(u + h - (-\Delta)\zeta^\phi (-\Delta)^{-1}\Phi^\phi(u + h)),
\\
\D M^\phi(u + h - \Phi_\psi(u + h))
&=
\D M^\phi(u + h - (-\Delta)\zeta_\psi (-\Delta)^{-1}\Phi_\psi(u + h)),
\end{aligned}
\end{equation}
for all $h \in L^p(\Omega)$ with $\norm{h}_{L^p(\Omega)} \leq r$.
\end{lemma}

\begin{proof}
As $Z_\psi$ and $Z^\phi$ are nonempty, compact, and disjoint subsets of $\Omega$,
we can find $\zeta_\psi \in C_c^\infty(\Omega)$ 
	such that $0 \le \zeta_\psi \le 1$ in $\Omega$,
	$\zeta_\psi \equiv 1$ on a neighborhood of $Z^\phi$,
	and
	$\zeta_\psi \equiv 0$ on a neighborhood of $Z_\psi$.
	Suppose that $h \in L^p(\Omega)$
	satisfying $\norm{h}_{L^p(\Omega)} \le r$ is given.
	We define \smash{$y(h) := S_\psi^\phi(u + h)$,}
	$\lambda_\psi(h) := \Phi_\psi(u + h)$,
	and
	$\lambda^\phi(h) := \Phi^\phi(u + h)$.
	Then it follows from \eqref{eq:PhiDef} that 
	$
		-\Delta y(h) = u + h - \lambda_\psi(h) - \lambda^\phi(h).
	$
	Let $\tilde y(h) \in H_0^1(\Omega)$ be the solution of 
	\begin{align*}
		-\Delta \tilde y(h) = u + h - (-\Delta)\zeta_\psi(-\Delta)^{-1}\lambda_\psi(h) - \lambda^\phi(h).
	\end{align*}
	Then it holds
	\begin{align*}
		y(h) - \tilde y(h)
		=
		- (1-\zeta_\psi) (-\Delta)^{-1} \lambda_\psi(h)
		.
	\end{align*}
	In particular, $y(h) = \tilde y(h)$ a.e.\ on $Z^\phi$ since $\zeta_\psi \equiv 1$ on $Z^\phi$.
	Moreover, the comparison principle 
	and the inequality $\lambda_\psi(h) \le 0$ a.e.\ in $\Omega$ imply 
	that $y(h) - \tilde y(h) \ge 0$ a.e.\ in $\Omega$. Thus, $\tilde y(h) \leq y(h) \leq \phi$ a.e.\ in $\Omega$.
	From \eqref{eq:randomsys26365465-1d}, we further get $(y(h) - \phi) \lambda^\phi(h) = 0 $
	a.e.\ in $\Omega$.
	Since $\set{y(h) = \phi} \subset Z^\phi$, $y(h) = \tilde y(h)$ a.e.\ on $Z^\phi$,
	and $\supp(\lambda^\phi(h) ) \subset \set{y(h) = \phi}$,
	this implies
	$(\tilde y(h) - \phi) \lambda^\phi(h) = 0 $
	a.e.\ on $\Omega$.
	Thus,
	\begin{align*}
		\lambda^\phi(h) \geq 0 \text{ a.e.\ in }\Omega,
		\qquad 
		\tilde y(h) - \phi \leq 0 \text{ a.e.\ in } \Omega,
		\qquad
		(\tilde y(h) - \phi)\lambda^\phi(h) = 0 \text{ a.e.\ in } \Omega.
	\end{align*}
	Together with the definition of $\tilde y(h)$,
	it follows that
	\begin{equation}
	\label{eq:tilde y}
	\begin{aligned}
		\tilde y(h)  &= S^\phi( u + h - (-\Delta)\zeta_\psi(-\Delta)^{-1}\lambda_\psi(h) ),
		\\
		\lambda^\phi(h) &= M^\phi( u + h - (-\Delta)\zeta_\psi(-\Delta)^{-1}\lambda_\psi(h) ).
	\end{aligned}	
	\end{equation}
	Due to the definitions $\lambda_\psi(h) := \Phi_\psi(u + h)$
	and
	$\lambda^\phi(h) := \Phi^\phi(u + h)$
	and the identity 
	\begin{align*}
	\Phi^\phi(u + h) = 
	\max(0, M_\psi^\phi(u+h)) = M^\phi \left ( u + h - \Phi_\psi(u + h) \right )
	\end{align*}
	obtained from \eqref{eq:PhiDef} and \cref{prop:impsys}, this proves the 
	second line of \eqref{eq:Midentities}.
	Since $\set{y(h) = \phi} \subset Z^\phi$, $y(h) = \tilde y(h)$ on $Z^\phi$,
	and $\tilde y(h) \leq y(h) \leq \phi$ in $\Omega$,
	we further have $\set{y(h) = \phi} = \set{\tilde y(h) = \phi}$
	and, as a consequence, $\set{y(h) < \phi} = \set{\tilde y(h) < \phi}$.
	In view of \cref{def:GenDerUni}, \eqref{eq:tilde y}, and again \cref{prop:impsys}, this yields
	\begin{align*}
	\D S^\phi(u + h - \lambda_\psi(h) )
	& = 
	\big \{
\SSS(O)
\,\left |~
O \text{ open and }
\set{y(h) < \phi} \subset O \subset \Omega \setminus \supp\left (\lambda^\phi(h)\right )
\big \}\right.
\\
	&= \big \{
\SSS(O)
\,\left |~
O \text{ open and }
\set{\tilde y(h) < \phi} \subset O \subset \Omega \setminus \supp\left (\lambda^\phi(h)\right )
\big \}\right.
\\
	& = 
	\D S^\phi(u + h - (-\Delta)\zeta_\psi(-\Delta)^{-1}\lambda_\psi(h)).
	\end{align*}
	Due to \cref{def:GenDerUni}, this establishes the fourth line of \eqref{eq:Midentities}.
	The assertions for $M_\psi$ in the first and the third line of \eqref{eq:Midentities} 
	are proved completely analogously. 
\end{proof}

By combining \cref{lem:regularity,lem:reg_bootstrap}, we now obtain 
\eqref{eq:weird_fraction} and, ultimately, the Newton differentiability of $\Phi$
as a function from $L^p(\Omega)$ to $\HH$.

\begin{proposition}%
\label{prop:semismoothPhi}%
The function 
$\Phi$
in \eqref{eq:PhiDef}
is Newton differentiable 
as a map from $L^p(\Omega)$
to
$\HH$
with Newton derivative 
\begin{align}
\label{eq:DPhiDef}
\D\Phi\colon L^p(\Omega)
\rightrightarrows
\LL(L^p(\Omega),\HH),
\quad
u \mapsto
 \{ - \GG_1^{-1}\GG_2 \mid \GG \in \D F(\Phi(u),u)\}.
\end{align}
\end{proposition}
\begin{proof}
Let $u \in L^p(\Omega)$ be given
and let $r > 0$ be chosen such that the assertions of 
\cref{lem:sepcontact,lem:Tinversion,lem:reg_bootstrap}
hold for some 
$Z_\psi, Z^\phi\subset \Omega$,  $C>0$, and  $\zeta_\psi, \zeta^\phi \in C_c^\infty(\Omega)$.
Suppose further that $h \in L^p(\Omega)$ with $0 < \|h\|_{L^p(\Omega)} \leq r$ is 
given.
Then we obtain from the arguments in the first half of 
\eqref{eq:longequation22}
and \cref{def:def.4.12} that
\begin{equation}
\label{eq:randomeq3737}
\begin{aligned}
&
\sup_{G \in \D \Phi(u + h)} 
	\frac{\norm{ \Phi(u + h) - \Phi(u) - G h }_\HH}{\norm{h}_{L^p(\Omega)}}
	\\&
	\leq
	C  \sup_{\GG \in \D F(\Phi(u+h),u+h)} 
	\frac{\norm{ \GG_1 (\Phi(u + h) - \Phi(u) )+ \GG_2 h }_\HH}{\norm{h}_{L^p(\Omega)}}
	\\&
	=
	C \hspace{-0.15cm} \sup_{\substack{
G_\psi \in D M_\psi(u +h  - \Phi^\phi(u + h)),
\\ 
G^\phi \in D M^\phi(u +h  - \Phi_\psi(u + h))}}\hspace{-0.15cm}
	\frac{\left \| 
		\begin{pmatrix}
			\Phi_\psi(u + h) - \Phi_\psi(u) +  G_\psi(\Phi^\phi(u + h) - \Phi^\phi(u) - h  )
			\\
			\Phi^\phi(u + h) - \Phi^\phi(u) +  G^\phi(\Phi_\psi(u + h) - \Phi_\psi(u) - h  )
		\end{pmatrix}	
	\right \|_\HH}
	{\norm{h}_{L^p(\Omega)}}
	\end{aligned}
\end{equation}
holds.
Let us again define $\lambda^\phi(h):= \Phi^\phi(u + h)$, $\lambda^\phi(0) := \Phi^\phi(u)$,
and  \smash{$y(h) := S_\psi^\phi(u + h)$}.
Due to
\cref{lem:regularity,lem:reg_bootstrap}, \eqref{eq:PhiDef},
and
\cref{prop:impsys},
we have 
\begin{equation*}
\Phi_\psi (u + h) - \Phi_\psi(u)
= M_\psi( u + h - \lambda^\phi(h)) - M_\psi(u - \lambda^\phi(0))  
=
     M_\psi( \tilde u + \tilde g(h)) - M_\psi(\tilde u )  
\end{equation*}
and
\begin{align*}
\D M_\psi(u + h - \lambda^\phi(h))
=
\D M_\psi(\tilde u + \tilde g(h))
\end{align*}
with 
\begin{align*}
\tilde u &:= u -  (-\Delta)\zeta^\phi(-\Delta)^{-1}\lambda^\phi(0) \in L^p(\Omega),
\\
\tilde g(h) &:=
h + 
(-\Delta)\zeta^\phi(-\Delta)^{-1} ( \lambda^\phi(0) - \lambda^\phi(h)) \in L^p(\Omega).
\end{align*}
For every $G_\psi \in \D M_\psi(u + h - \lambda^\phi(h))$ with 
generating set $O_\psi$ as in \cref{def:GenDerUni},
we further obtain from  $1 - \zeta^\phi \equiv 0$ on 
$Z_\psi$, $\{y(h) > \psi\} \subset O_\psi$,
and $\{y(h) = \psi\} \subset Z_\psi$ that
$( 1 - \zeta^\phi)(-\Delta)^{-1} ( \lambda^\phi(h) - \lambda^\phi(0)) \in H_0^1(O_\psi)$ and, thus, 
\begin{align*}
 \SSS(O_\psi)(-\Delta)( 1 - \zeta^\phi)(-\Delta)^{-1} ( \lambda^\phi(h)  - \lambda^\phi(0) )
 = ( 1 - \zeta^\phi)(-\Delta)^{-1} ( \lambda^\phi(h)  - \lambda^\phi(0) ).
\end{align*}
Due to $G_\psi = \Id - (-\Delta)\SSS(O_\psi)$,  this yields 
\begin{align*}
G_\psi( \lambda^\phi(h) - \lambda^\phi(0) - h)
&=
G_\psi(-\Delta)\zeta^\phi(-\Delta)^{-1}(\lambda^\phi(h) - \lambda^\phi(0)) - G_\psi h
\\
&=
- G_\psi  \tilde g(h). 
\end{align*}
If we put all of the above together, then it follows that 
\begin{equation}
\label{eq:randomeq273uzh}
\begin{aligned}
&\sup_{
G_\psi \in D M_\psi(u +h  - \Phi^\phi(u + h))} 
 \left \| 
\Phi_\psi (u + h) - \Phi_\psi(u)  +  G_\psi\left (\Phi^\phi(u + h) - \Phi^\phi(u) - h\right ) 
	\right \|_{H^{-1}(\Omega)}
	\\
&=
\sup_{
G_\psi \in \D M_\psi(\tilde u + \tilde g(h))} 
 \left \| 
  M_\psi( \tilde u + \tilde g(h)) - M_\psi(\tilde u )  
  -
 G_\psi\tilde g(h)
	\right \|_{H^{-1}(\Omega)}.
\end{aligned}
\end{equation}
Note that, due to the Lipschitz continuity of 
$\Phi^\phi$ from  
$L^p(\Omega)$ to $L^1(\Omega)$ (see \cref{lem:regularity_again} and \eqref{eq:PhiDef}),
the fact that 
$\lambda^\phi(h) - \lambda^\phi(0)
=
\Phi^\phi(u + h) - \Phi^\phi(u)$ is supported only on $Z^\phi$,
and \cref{lem:regularity},
  we know that there exist  constants $C_1, C_2 > 0$ such that 
\begin{align*}
\|\tilde g(h)\|_{L^p(\Omega)}
&\leq
\| h \|_{L^p(\Omega)}
+
\| (-\Delta)\zeta^\phi(-\Delta)^{-1} ( \lambda^\phi(0) - \lambda^\phi(h)) \|_{L^p(\Omega)}
\\
&\leq 
\| h \|_{L^p(\Omega)}
+
C_1
\| \lambda^\phi(0) - \lambda^\phi(h) \|_{L^1(\Omega)}
\\
&\leq
(1 + C_1 C_2)
\| h \|_{L^p(\Omega)}.
\end{align*}
In combination with the Newton differentiability of 
$M_\psi$ as a function from $L^p(\Omega)$ to $H^{-1}(\Omega)$,
\eqref{eq:randomeq273uzh}, and the fact that 
$h$ was arbitrary, 
 this allows us to conclude that 
\begin{equation*}
\begin{aligned}
&\sup_{
G_\psi \in D M_\psi(u +h  - \Phi^\phi(u + h))} 
 \frac{\left \| 
\Phi_\psi (u + h) - \Phi_\psi(u)  +  G_\psi\left (\Phi^\phi(u + h) - \Phi^\phi(u) - h\right ) 
	\right \|_{H^{-1}(\Omega)}}{\|h\|_{L^p(\Omega)}}
	\\
&\leq 
(1 + C_1 C_2)
\sup_{
G_\psi \in \D M_\psi(\tilde u + \tilde g(h))} 
 \frac{\left \| 
  M_\psi( \tilde u + \tilde g(h)) - M_\psi(\tilde u )  
  -
 G_\psi\tilde g(h)
	\right \|_{H^{-1}(\Omega)}}{\|\tilde g(h)\|_{L^p(\Omega)}}
	\to 0
\end{aligned}
\end{equation*}
holds for $(0, r] \ni \|h\|_{L^p(\Omega)} \to 0$, where, in the case $\tilde g(h) = 0$,
the second fraction is to be understood as zero. 
By proceeding analogously with the second component in \eqref{eq:randomeq3737}
and by plugging everything in, we now obtain that
\[
\sup_{G \in \D \Phi(u + h)} 
	\frac{\norm{ \Phi(u + h) - \Phi(u) - G h }_\HH}{\norm{h}_{L^p(\Omega)}}
	\to 0
	\quad\text{ for } 0 <  \|h\|_{L^p(\Omega)} \to 0
\]
holds as desired. This completes the proof. 
\end{proof}

As $\Phi$ is the function that maps $u \in L^p(\Omega)$
to the positive and the negative part of the multiplier 
\smash{$M_\psi^\phi(u) \in L^p(\Omega) \subset H^{-1}(\Omega)$}
and since the description of 
$\D \Phi$ in \eqref{eq:DPhiDef} is cumbersome, 
\cref{prop:semismoothPhi} is not well suited for applications. 
To refine it, we can use \eqref{eq:Pdef}. Doing so yields the following 
second main result of this paper which shows that
the solution mapping \smash{$S_\psi^\phi$}
of the bilateral obstacle problem is Newton differentiable 
with a nicely describable Newton derivative in the situation of \cref{ass:sec6}.

\begin{theorem}[$H_0^1$-Newton differentiability for the bilateral obstacle problem]%
	\label{th:semismoothH01bilateral}%
	The operator  $S_\psi^\phi$ is Newton differentiable 
	as a function $S_\psi^\phi\colon L^p(\Omega) \to H_0^1(\Omega) \cap L^q(\Omega)$
	for all $1 \leq q < \infty$ with the derivative  
	$\smash{\D S_\psi^\phi\colon L^p(\Omega) \rightrightarrows} \LL(L^p(\Omega), H_0^{1}(\Omega))$
	given by
\begin{align}
\label{eq:DSphipsiDef}
\D S_\psi^\phi(u) := 
\Big \{
\SSS(O)
\, \Big |~
O &\text{ open and } 
\set{\psi < S_\psi^\phi(u) < \phi } \subset O \subset \Omega \setminus \supp\big (M_\psi^\phi(u)\big ) 
\Big\}.
\end{align}
Here, the inactive set 
is
defined w.r.t.\ the 
$C(\Omega)$-representatives of $S_\psi^\phi(u)$, $\psi$, and $\phi$,
and the symbol $\supp(\cdot)$ denotes the support in the sense of Borel measures on $\Omega$.
\end{theorem}

\begin{proof}
From \cref{prop:semismoothPhi} and \eqref{eq:PhiDef}, we obtain that 
\smash{$M_\psi^\phi\colon L^p(\Omega) \to H^{-1}(\Omega)$} is Newton differentiable 
with derivative
\begin{align}
\label{eq:DMpisphiDef}
\D M_\psi^\phi(u)
:=
\left \{ - \begin{pmatrix}
			\Id & \Id
\end{pmatrix}\GG_1^{-1}\GG_2 \mid \GG \in \D F(\Phi(u),u)
\right \}.
\end{align}
Suppose now that an open set $O$ satisfying the conditions in 
\eqref{eq:DSphipsiDef} is given and define
\begin{align}
\label{defOpsiOphiBi}
O_\psi := O \cup \set{ S_\psi^\phi(u) > \psi}
\qquad
\text{and}
\qquad 
 O^\phi := O \cup \set{ S_\psi^\phi(u) < \phi}.
\end{align}
Then $O_\psi$ and $O^\phi$  satisfy the conditions in \cref{def:GenDerUni}
and give rise to operators 
$G_\psi \in \D M_\psi(u - \Phi^\phi(u))$ 
and 
$G^\phi \in \D M^\phi (u - \Phi_\psi(u))$, respectively;
cf.\ \cref{prop:impsys}.

Due to \cref{lem:projidentification}, 
$G_\psi$ and $G^\phi$ are precisely
the orthogonal projections 
\smash{$P_{W(G_\psi)}$} and \smash{$P_{W(G^\phi)}$}
onto the spaces 
$W(G_\psi)$
and 
$W(G^\phi)$ in \eqref{eq:Wdefs}.
In combination with \cref{def:def.4.12}, \eqref{eq:DMpisphiDef},
\cref{lem:boundproj}, and \cref{thm:some_invertibility},
 this yields 
	\begin{align*}
		\D M_\psi^\phi(u)
		&\ni 
		\begin{pmatrix}
			\Id & \Id
\end{pmatrix}
		\begin{pmatrix}
			\Id & P_{W(G_\psi)}
			\\
			P_{W(G^\phi)} & \Id
		\end{pmatrix}^{-1}
		\begin{pmatrix}
			P_{W(G_\psi)}
			\\
			P_{W(G^\phi)}
		\end{pmatrix} = P_{W(G_\psi) + W(G^\phi)},
	\end{align*}
	where \smash{$P_{W(G_\psi) + W(G^\phi)}$}
	denotes the orthogonal projection onto $W(G_\psi) + W(G^\phi)$.
	From \eqref{eq:Wdefs}, \cite[Theorem~4.5]{HeinonenKilpelaeinenMartio1993},
	and \eqref{defOpsiOphiBi}, it follows further that
	\begin{align*}
		W(G_\psi) + W(G^\phi)
		&=
		\bracks*{- \Delta H_0^1(O_\psi) }^\perp
		+
		\bracks*{- \Delta H_0^1(O^\phi) }^\perp
		\\
		&=
		\bracks*{- \Delta \parens*{H_0^1(O_\psi) \cap H_0^1(O^\phi)} }^\perp
		\\
		&=
		\bracks*{- \Delta  H_0^1(O_\psi \cap O^\phi) }^\perp
		\\
		&= \bracks*{- \Delta H_0^1(O) }^\perp.
	\end{align*}
       This shows that \smash{$\Id - P_{W(G_\psi) + W(G^\phi)}$} is the orthogonal 
       projection onto \smash{$- \Delta  H_0^1(O)$} and, in view of \eqref{eq:H-1scalarproduct}
       and \cref{def:PoissonSolutionOps}, that
       \smash{$(-\Delta)^{-1}( \Id - P_{W(G_\psi) + W(G^\phi)})= \SSS(O)$}.
       If we combine all of the above, then we obtain that
       $\SSS(O) \in (-\Delta)^{-1}( \Id - \D M_\psi^\phi(u))$.
       Since \smash{$S_\psi^\phi = (-\Delta)^{-1}(\Id - M_\psi^\phi)$}
       holds by
       the definition of \smash{$M_\psi^\phi$}, the Newton differentiability 
       of the operator \smash{$S_\psi^\phi $} as a function from $L^p(\Omega)$ to $H_0^1(\Omega)$
       with the derivative in \eqref{eq:DSphipsiDef} now follows immediately 
       from the Newton differentiability properties of \smash{$M_\psi^\phi$}. It remains 
       to prove that the operator \smash{$S_\psi^\phi $} is also Newton differentiable as a function from 
       $L^p(\Omega)$ to $L^q(\Omega)$ for all $1 \leq q < \infty$. This, however, 
       follows straightforwardly from the definition of Newton differentiability,
       the fact that \smash{$S_\psi^\phi$} is globally Lipschitz from $L^p(\Omega)$
       to $L^\infty(\Omega)$ by \cref{lem:regularity_again}, 
       the fact that
       $\{\SSS(O) \mid O \subset \Omega \text{ nonempty, open}\}$ is a bounded subset of $\LL(L^p(\Omega), L^\infty(\Omega))$
       (which can easily be checked by means of \cref{def:PoissonSolutionOps} and the classical truncation 
       arguments in \cite[Theorem B.2]{KinderlehrerStampacchia1980}), 
       the Newton differentiability of \smash{$S_\psi^\phi$} as a map from $L^p(\Omega)$ to $H_0^1(\Omega)$,
	and 
       the 
       dominated convergence theorem; cf.\ the arguments in 
       the proof of \cref{th:semismoothH01unilateral}.~
\end{proof}

\begin{remark}~
\begin{itemize}
\item By reversing the calculations in the proof of \cref{th:semismoothH01bilateral}, one, of course, 
also obtains that the map \smash{$M_\psi^\phi$} is Newton differentiable as a 
function from $L^p(\Omega)$ to $H^{-1}(\Omega)$ with Newton derivative 
\smash{$\D M_\psi^\phi(u) := \Id - (-\Delta)\D S_\psi^\phi(u)$}.
\item The Newton derivative $\D S_\psi^\phi(u) $ in \eqref{eq:DSphipsiDef}
contains the elements of $\partial_B^{sw}S_\psi^\phi(u)$
identified in \cite[Theorem 6.2]{RaulsUlbrich2021} as special cases. 
Using the implicit function 
approach developed in this section, it might be possible to show that all 
elements of the set $\D S_\psi^\phi(u)$ in \eqref{eq:DSphipsiDef} are, in fact, generalized derivatives 
in the strong-weak Bouligand sense.  We leave this topic for future research.
\end{itemize}
\end{remark}

 \section{An application example from optimal control}
\label{sec:7}
We conclude this paper with an application example from optimal control which illustrates how \cref{th:semismoothH01bilateral}
can be used in practice. As the following analysis parallels that in \cite[sections 5, 6]{ChristofWachsmuth2023},
we do not aim for the highest level of generality and keep the discussion concise. 
The model problem that we consider reads as follows:
\begin{equation}
\label{eq:M}
\tag{\textup{M}}
	\left \{\,\,
	\begin{aligned}
		\text{Minimize} \quad & \frac12 \norm*{ y - y_D}_{L^2(\Omega)}^2 + 
		\frac{\nu}{2}\|u\|_{H_0^1(\Omega)}^2\\
		\text{w.r.t.}\quad &u, y \in H_0^1(\Omega), \\
		\text{s.t.}\quad & \mathopen{}-\Delta y  = u \text{ and } \psi \leq u \leq \phi \text{ in }\Omega.
	\end{aligned}
	\right.
\end{equation}
\begin{assumption}[standing assumptions for \cref{sec:7}]~
\begin{itemize}
\item $\Omega \subset \R^d$, $d \in \{1,2,3\}$, is a bounded Lipschitz domain;
\item $\nu > 0$ and $y_D \in L^2(\Omega)$ are given;
\item $\psi$ and $\phi$ satisfy
$\psi, \phi \in C(\bar\Omega) \cap H^1(\Omega)$ and $\Delta\psi, \Delta\phi \in L^2(\Omega)$, 
and there exists a number $c > 0$ such that \eqref{eq:constant_c} holds.
\end{itemize}
\end{assumption}

Note that the standing assumptions of \cref{sec:6} hold in the above situation with $p=2$; 
cf.\ \cref{rem:sec6:LipschitzDomain}.
From standard arguments, we obtain the following theorem.

\begin{theorem}
Problem \eqref{eq:M} possesses a unique optimal control $\bar u$ with associated 
state $\bar y := (-\Delta)^{-1}\bar u$. The optimal state $\bar y$ is uniquely characterized by the equation
\begin{equation}
\label{eq:MFONC}
\bar y - (-\Delta)^{-1} S_\psi^\phi \parens*{\nu^{-1} (-\Delta)^{-1}( y_D - \bar y)} = 0,
\end{equation}
where $S_\psi^\phi\colon H^{-1}(\Omega) \to H_0^1(\Omega)$ denotes the solution map of \eqref{eq:VI}.
\end{theorem}
\begin{proof}
The unique solvability follows from the direct method of the calculus of variations and 
the strict convexity of \eqref{eq:M}. Equation \eqref{eq:MFONC} is obtained from the exact same arguments 
as in \cite[section 5]{ChristofWachsmuth2023}; see \cite[Proposition 5.2, Equation (5.4)]{ChristofWachsmuth2023}.
\end{proof}

The next lemma shows that \eqref{eq:MFONC} is amenable to a semismooth Newton method.

\begin{lemma}%
\label{lem:leMMa}%
 Let $\D S_\psi^\phi$ be defined as in \eqref{eq:DSphipsiDef}. The function 
\begin{align*}
Q\colon L^2(\Omega) \to L^2(\Omega),
\qquad
y \mapsto  y - (-\Delta)^{-1} S_\psi^\phi \parens*{\nu^{-1} (-\Delta)^{-1}( y_D - y)},
\end{align*}
is Newton differentiable with derivative 
\begin{align*}
\D Q(y) := \Id + \nu^{-1}(-\Delta)^{-1} \D S_\psi^\phi\left [\nu^{-1} (-\Delta)^{-1}(y_D - y) \right ](-\Delta)^{-1} \subset 
\LL(L^2(\Omega), L^2(\Omega)).
\end{align*}
Further, for all $y \in L^2(\Omega)$ and $G \in \D Q(y)$, 
$G^{-1}$ exists and $\|G^{-1}\|_{\LL(L^2(\Omega), L^2(\Omega))} \leq 1$.
\end{lemma}
\begin{proof}
The Newton differentiability of $Q$ as a function from $L^2(\Omega)$ to $L^2(\Omega)$
with derivative $\D Q$
follows from the estimate $2> \max(1,d/2)$,
the linearity and continuity of 
$(-\Delta)^{-1}\colon L^2(\Omega) \to H_0^1(\Omega) \hookrightarrow L^2(\Omega)$,
and \cref{th:semismoothH01bilateral}.
To see that the elements of $\D Q(y)$
are invertible for all $y \in L^2(\Omega)$, we note that \eqref{eq:DSphipsiDef} and \cref{def:PoissonSolutionOps}  imply  that,
for every $G \in \D Q(y)$, there exists an open nonempty set $O \subset \Omega$ such that 
\begin{align*}
\left (z, Gz  \right )_{L^2(\Omega)}
&=
\|z\|_{L^2(\Omega)}^2
+
\nu^{-1}
\left (z, (-\Delta)^{-1} \SSS(O)(-\Delta)^{-1} z \right )_{L^2(\Omega)}
\\
&=
\|z\|_{L^2(\Omega)}^2
+
\nu^{-1}
\left ((-\Delta)^{-1} z,  \SSS(O)(-\Delta)^{-1} z \right )_{L^2(\Omega)}
\\
&=
\|z\|_{L^2(\Omega)}^2
+
\nu^{-1}
\left (\SSS(O)(-\Delta)^{-1} z,  \SSS(O)(-\Delta)^{-1} z \right )_{H_0^1(\Omega)}
\\
&\geq 
\|z\|_{L^2(\Omega)}^2\qquad \forall z \in L^2(\Omega).
\end{align*}
In combination with the lemma of Lax--Milgram, this establishes the claim. 
\end{proof}

The fact that the lemma of Lax--Milgram can be used in the situation of \eqref{eq:MFONC} 
to establish the invertibility of the Newton derivatives is the main reason why we have formulated the 
first-order optimality condition of \eqref{eq:M} in $y$ and not in $u$.  
\Cref{lem:leMMa} shows that it makes sense to consider a semismooth Newton method for 
the solution of \eqref{eq:MFONC}. The resulting algorithm can be seen in \cref{alg:semiNewton}. 

\begin{Algorithm}[semismooth Newton method for the solution of \eqref{eq:M}]\label{alg:semiNewton}
~\hspace{-10cm}
\begin{algorithmic}[1]
  \STATE Choose an initial guess $y_0 \in L^2(\Omega)$ and a tolerance $\texttt{tol} \geq 0$.
    \FOR{$i = 0,1,2,3,\ldots$}
    \STATE Calculate $z_i :=(-\Delta)^{-1}( y_D - y_i)/\nu$, $u_i := S_\psi^\phi (z_i)$, and $\tilde y_i := (-\Delta)^{-1}u_i$.
        \IF{$\|y_i - \tilde y_i\|_{L^2(\Omega)} \leq \texttt{tol}$}
            \STATE STOP the iteration (convergence reached). 
        \ELSE
        \STATE 
           \label{algo1:step:7}
            Choose an open set $O$ 
            satisfying 
            $\set{\psi < S_\psi^\phi(z_i) < \phi } \subset O \subset \Omega \setminus \supp\big (M_\psi^\phi(z_i)\big )$ 
            with \smash{$S_\psi^\phi$} and \smash{$M_\psi^\phi$} as in \cref{th:semismoothH01bilateral}
            and determine $y_{i+1} \in L^2(\Omega)$ by solving 
        \begin{align*}
        	y_{i + 1} + \nu^{-1} (-\Delta)^{-1}\SSS(O)(-\Delta)^{-1}y_{i+1} 
        	= \tilde y_i + \nu^{-1}  (-\Delta)^{-1}\SSS(O)(-\Delta)^{-1}y_i.
        \end{align*}
        \ENDIF
    \ENDFOR
\end{algorithmic}
\end{Algorithm}

From the standard convergence analysis for the semismooth Newton method, 
we now obtain the following result.

\begin{theorem}[local $q$-superlinear convergence]%
\label{th:NewtonConvergence}%
Let $\bar u$ be the optimal control of \eqref{eq:M} and $\bar y := (-\Delta)^{-1}\bar u$ the associated optimal state.
There exists $r > 0$ such that, for every $y_0 \in L^2(\Omega)$ with 
$\|y_0 - \bar y\|_{L^2(\Omega)} < r$,
\cref{alg:semiNewton} with $\texttt{tol} = 0$ either terminates after finitely many steps with $\bar y$ or 
produces a sequence $\{y_i\} \subset L^2(\Omega)$ that converges 
$q$-superlinearly to $\bar y$ in $L^2(\Omega)$. 
\end{theorem}
\begin{proof}
	The update step  in \cref{alg:semiNewton} is equivalent to
	$G_i (y_{i+1} - y_i) = - Q(y_i)$ with $G_i \in \D Q(y_i)$ and with $Q$ and $\D Q$ as in \cref{lem:leMMa}.
	From \cref{lem:leMMa}, we further know that $Q$ is Newton differentiable from 
	$L^2(\Omega)$ to $L^2(\Omega)$ with derivative $\D Q$ and that all  $G$ in the image of $\D Q$ are
	invertible with $\|G^{-1}\|_{\LL(L^2(\Omega), L^2(\Omega))} \leq 1$. 
	With this information at hand,
	the local $q$-superlinear convergence of the iterates $\{y_i\}$ in $L^2(\Omega)$ to $\bar y$ follows 
	completely analogously to, e.g., the proof of
	\cite[Theorem 3.4]{ChenNashedQi2000}.~
\end{proof}

The results that are obtained when piecewise linear finite elements
are used to discretize  
\eqref{eq:M} 
and a semismooth Newton method along the lines of \cref{alg:semiNewton} is applied to 
the resulting finite-dimensional problem (analogously to \cite[section~6, Algorithm~6.1]{ChristofWachsmuth2023})
can be seen in \cref{tab:1} and \cref{fig:1}. Here, we considered the situation
$\Omega = (0,1)^2$, 
$\nu = 10^{-5}$, $y_D(x_1, x_2) = 10(-x_1 - x_2 + 1)$,  $\psi(x_1, x_2) = -5$,  $\phi(x_1, x_2) = 5$,  
$y_0 \equiv 0$, and $\texttt{tol} = 10^{-12}$;
the underlying meshes were Friedrichs--Keller triangulations of width $h$; 
and the inactive set was used as the set $O$
in step \ref{algo1:step:7} of the algorithm.
As \cref{tab:1} shows, 
the number of iterations that our Newton method requires to drive the residue below
 $\texttt{tol} = 10^{-12}$ does not change in our experiment as $h$ tends to zero. 
This so-called mesh-independence is characteristic for algorithms
whose convergence can be proved not only in the discrete but also in the infinite-dimensional setting.
\begin{table}[H]
\caption{Number of Newton iterations needed to drive the $L^2(\Omega)$-residue below $\texttt{tol} = 10^{-12}$
for various $h$.}
\label{tab:1}
\centering
\begin{tabular}{c |c | c | c | c | c | c | c}
\noalign{\smallskip}
mesh width $h$ &  $\frac{1}{32}$ & $\frac{1}{64}$  & $\frac{1}{128}$   & $\frac{1}{256}$   & $\frac{1}{512}$  & $\frac{1}{1024}$  \\
\noalign{\smallskip}\hline\noalign{\smallskip}
number of iterations  & $2$ & $2$ & $2$ &  $2$ & $2$ & $2$
\end{tabular}\pagebreak

\end{table}
\begin{figure}[ht]
\centering 
\subfigure{\includegraphics[width=5.9cm,height=7cm,keepaspectratio]{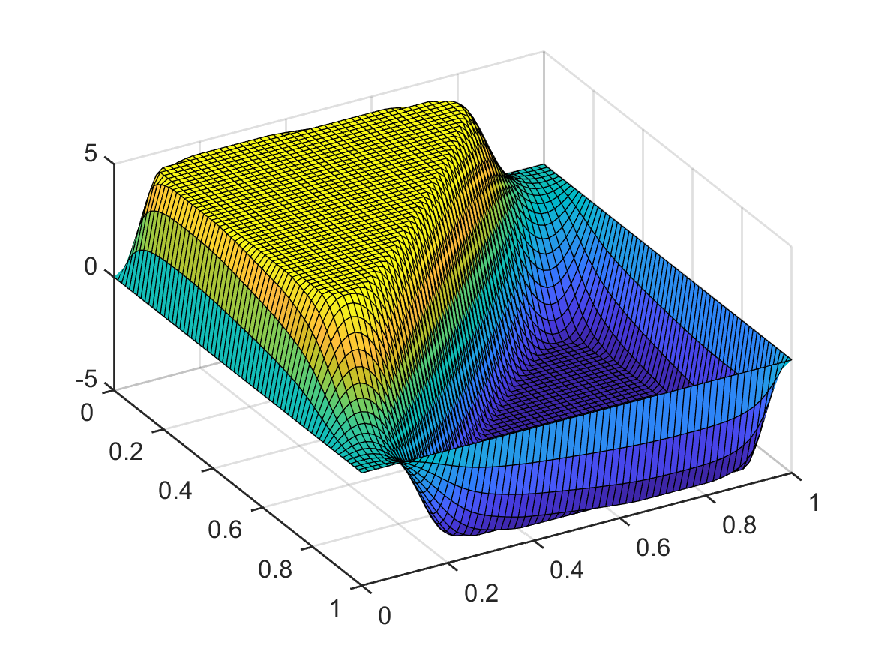}} %
\hspace{0.5cm}%
\subfigure{\includegraphics[width=5.9cm,height=7cm,keepaspectratio]{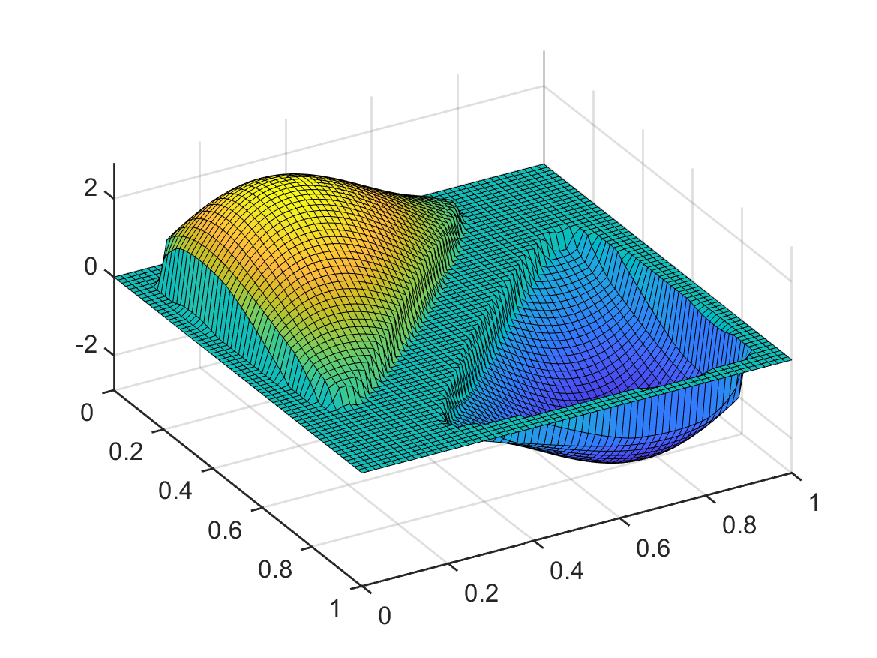}} %
\caption{Optimal control (left) and associated
Lagrange multiplier of the bi-obstacle problem (right) obtained from a finite element discretization of 
\cref{alg:semiNewton} with mesh width $h = 1/64$. }
\label{fig:1}
\end{figure}


\bibliographystyle{siamplain}
\bibliography{references}

\begin{thebibliography}{10}

\bibitem{Attouch2006}
{\sc H.~Attouch, G.~Buttazzo, and G.~Michaille}, {\em Variational Analysis in
  Sobolev and BV Spaces}, MPS/SIAM Series on Optimization, SIAM, Philadelphia,
  2006, \url{https://doi.org/10.1137/1.9781611973488}.

\bibitem{Bartels2015}
{\sc S.~Bartels}, {\em Numerical Methods for Nonlinear Partial Differential
  Equations}, no.~47 in Springer Series in Computational Mathematics, Springer
  International Publishing, Cham, 2015,
  \url{https://doi.org/10.1007/978-3-319-13797-1}.

\bibitem{BonnansShapiro2000}
{\sc J.~F. Bonnans and A.~Shapiro}, {\em Perturbation Analysis of Optimization
  Problems}, Springer Series in Operations Research, Springer, New York, 2000.

\bibitem{Brokate2020}
{\sc M.~Brokate}, {\em {N}ewton and {B}ouligand derivatives of the scalar play
  and stop operator}, Math.~Model.~Nat.~Phenom., 15 (2020),
  \url{https://doi.org/10.1051/mmnp/2020013}.

\bibitem{BrokateUlbrich2022}
{\sc M.~Brokate and M.~Ulbrich}, {\em Newton differentiability of convex
  functions in normed spaces and of a class of operators}, SIAM J.~Optim., 32
  (2022), pp.~1265--1287, \url{https://doi.org/10.1137/21M1449531}.

\bibitem{ChenNashedQi2000}
{\sc X.~Chen, Z.~Nashed, and L.~Qi}, {\em Smoothing methods and semismooth
  methods for nondifferentiable operator equations}, SIAM J.~Numer.~Anal., 38
  (2000), pp.~1200--1216, \url{https://doi.org/10.1137/s0036142999356719}.

\bibitem{ChristofPhd2018}
{\sc C.~Christof}, {\em Sensitivity Analysis of Elliptic Variational
  Inequalities of the First and the Second Kind}, PhD thesis, Technische
  Universit{\"a}t Dortmund, 2018.

\bibitem{Christof2018}
{\sc C.~Christof, C.~Meyer, S.~Walther, and C.~Clason}, {\em Optimal control of
  a non-smooth semilinear elliptic equation}, Math.~Control Relat.~Fields, 8
  (2018), pp.~247--276, \url{https://doi.org/10.3934/mcrf.2018011}.

\bibitem{ChristofWachsmuth2020}
{\sc C.~Christof and G.~Wachsmuth}, {\em Differential sensitivity analysis of
  variational inequalities with locally {L}ipschitz continuous solution
  operators}, Appl.~Math.~Optim., 81 (2020), pp.~23--62,
  \url{https://doi.org/10.1007/s00245-018-09553-y}.

\bibitem{ChristofWachsmuth2021SSC}
{\sc C.~Christof and G.~Wachsmuth}, {\em On second-order optimality conditions
  for optimal control problems governed by the obstacle problem}, Optimization,
  70 (2021), pp.~2247--2287,
  \url{https://doi.org/10.1080/02331934.2020.1778686}.

\bibitem{ChristofWachsmuth2023}
{\sc C.~Christof and G.~Wachsmuth}, {\em Semismoothness for solution operators
  of obstacle-type variational inequalities with applications in optimal
  control}, SIAM J.~Control Optim., 61 (2023), pp.~1162--1186,
  \url{https://doi.org/10.1137/21M1467365}.

\bibitem{Reyes2005}
{\sc J.~C. De~los Reyes and K.~Kunisch}, {\em A semi-smooth {N}ewton method for
  control constrained boundary optimal control of the {N}avier-{S}tokes
  equations}, J.~Nonlinear Anal.~Optim., 62 (2005), pp.~1289--1316,
  \url{https://doi.org/10.1016/j.na.2005.04.035}.

\bibitem{Deutsch1995}
{\sc F.~Deutsch}, {\em The angle between subspaces of a {H}ilbert space}, in
  Approximation Theory, Wavelets and Applications, S.~P. Singh, ed., Springer
  Netherlands, Dordrecht, 1995, pp.~107--130,
  \url{https://doi.org/10.1007/978-94-015-8577-4_7}.

\bibitem{Facchinei2003}
{\sc F.~Facchinei and J.~S. Pang}, {\em Finite-Dimensional Variational
  Inequalities and Complementarity Problems, Vol.\ II}, Springer Series in
  Operations Research, Springer, New York, 2003,
  \url{https://doi.org/10.1007/b97544}.

\bibitem{Haraux1977}
{\sc A.~Haraux}, {\em How to differentiate the projection on a convex set in
  {H}ilbert space. {S}ome applications to variational inequalities},
  J.~Math.~Soc.~Japan, 29 (1977), pp.~615--631.

\bibitem{HeinonenKilpelaeinenMartio1993}
{\sc J.~Heinonen, T.~Kilpel{\"a}inen, and O.~Martio}, {\em Nonlinear Potential
  Theory of Degenerate Elliptic Equations}, Oxford Mathematical Monographs, The
  Clarendon Press Oxford University Press, New York, 1993.
\newblock Oxford Science Publications.

\bibitem{Hintermueller2002}
{\sc M.~Hintermüller, K.~Ito, and K.~Kunisch}, {\em The primal-dual active set
  strategy as a semismooth {N}ewton method}, SIAM J.~Optim., 13 (2002),
  pp.~865--888, \url{https://doi.org/10.1137/s1052623401383558}.

\bibitem{Hintermueller2008}
{\sc M.~Hintermüller, F.~Tröltzsch, and I.~Yousept}, {\em Mesh-independence
  of semismooth {N}ewton methods for {L}avrentiev-regularized state constrained
  nonlinear optimal control problems}, Numer.~Math., 108 (2008), pp.~571--603,
  \url{https://doi.org/10.1007/s00211-007-0134-6}.

\bibitem{Izmailov2014}
{\sc A.~F. Izmailov and M.~V. Solodov}, {\em Newton-Type Methods for
  Optimization and Variational Problems}, Springer Series in Operations
  Research and Financial Engineering, Springer International Publishing, 2014,
  \url{https://doi.org/10.1007/978-3-319-04247-3}.

\bibitem{KinderlehrerStampacchia1980}
{\sc D.~Kinderlehrer and G.~Stampacchia}, {\em An Introduction to Variational
  Inequalities and Their Applications}, vol.~31 of Classics in Applied
  Mathematics, SIAM, 2000, \url{https://doi.org/10.1137/1.9780898719451}.

\bibitem{Mannel2018}
{\sc F.~Mannel}, {\em Semismooth implicit functions}, J.~Convex Anal., 25
  (2018), pp.~595--622.

\bibitem{Mifflin1977}
{\sc R.~Mifflin}, {\em An algorithm for constrained optimization with
  semismooth functions}, Math.~Oper.~Res., 2 (1977), pp.~191--207,
  \url{https://doi.org/10.1287/moor.2.2.191}.

\bibitem{Mifflin1977-2}
{\sc R.~Mifflin}, {\em Semismooth and semiconvex functions in constrained
  optimization}, SIAM J.~Control Optim., 15 (1977), pp.~959--972,
  \url{https://doi.org/10.1137/0315061}.

\bibitem{Mignot1976}
{\sc F.~Mignot}, {\em Contrôle dans les inéquations variationelles
  elliptiques}, J.~Funct.~Anal., 22 (1976), pp.~130--185,
  \url{https://doi.org/10.1016/0022-1236(76)90017-3}.

\bibitem{OutrataBook1998}
{\sc J.~Outrata, M.~Kocvara, and J.~Zowe}, {\em Nonsmooth Approach to
  Optimization Problems with Equilibrium Constraints: Theory, Applications and
  Numerical Results}, Springer, Dordrecht, 1998,
  \url{https://doi.org/10.1007/978-1-4757-2825-5}.

\bibitem{Ponce2016}
{\sc C.~A. Ponce}, {\em Elliptic {PDEs}, Measures and Capacities}, vol.~23 of
  EMS Tracts Math., Z{\"u}rich: European Mathematical Society (EMS), 2016,
  \url{https://doi.org/10.4171/140}.

\bibitem{RaulsUlbrich2021}
{\sc A.-T. Rauls and S.~Ulbrich}, {\em On the characterization of generalized
  derivatives for the solution operator of the bilateral obstacle problem},
  SIAM J.~Control Optim., 59 (2021), pp.~3683--3707,
  \url{https://doi.org/10.1137/20m135916x}.

\bibitem{Rauls2020}
{\sc A.-T. Rauls and G.~Wachsmuth}, {\em Generalized derivatives for the
  solution operator of the obstacle problem}, Set-Valued Var.~Anal., 28 (2020),
  pp.~259--285, \url{https://doi.org/10.1007/s11228-019-0506-y}.

\bibitem{Rodrigues1987}
{\sc J.~Rodrigues}, {\em Obstacle Problems in Mathematical Physics},
  North-Holland, 1987.

\bibitem{Roesch2011}
{\sc A.~Rösch and D.~Wachsmuth}, {\em Semi-smooth {N}ewton method for an
  optimal control problem with control and mixed control-state constraints},
  Optim.~Methods Softw., 26 (2011), pp.~169--186,
  \url{https://doi.org/10.1080/10556780903548257}.

\bibitem{Schiela2008}
{\sc A.~Schiela}, {\em A simplified approach to semismooth {N}ewton methods in
  function space}, SIAM J.~Optim., 19 (2008), pp.~369--393,
  \url{https://doi.org/10.1137/060674375}.

\bibitem{SchrammZowe1992}
{\sc H.~Schramm and J.~Zowe}, {\em A version of the bundle idea for minimizing
  a nonsmooth function: Conceptual idea, convergence analysis, numerical
  results}, SIAM J.~Optim., 2 (1992), pp.~121--152,
  \url{https://doi.org/10.1137/0802008}.

\bibitem{Rehberg2015}
{\sc A.~F.~M. ter Elst and J.~Rehberg}, {\em {H}\"older estimates for
  second-order operators on domains with rough boundary}, Adv.~Difference Equ.,
  20 (2015), pp.~299--360, \url{https://doi.org/10.57262/ade/1423055203}.

\bibitem{Thibault1982}
{\sc L.~Thibault}, {\em On generalized differentials and subdifferentials of
  {L}ipschitz vector-valued functions}, Nonlinear Anal., 6 (1982),
  pp.~1037--1053, \url{https://doi.org/10.1016/0362-546x(82)90074-8}.

\bibitem{Troianiello1987}
{\sc G.~M. Troianiello}, {\em Elliptic Differential Equations and Obstacle
  Problems}, Springer, 1987.

\bibitem{Ulbrich2011}
{\sc M.~Ulbrich}, {\em Semismooth {N}ewton Methods for Variational Inequalities
  and Constrained Optimization Problems in Function Spaces}, SIAM, 2011,
  \url{https://doi.org/10.1137/1.9781611970692}.

\bibitem{Wachmuth2014}
{\sc G.~Wachsmuth}, {\em Strong stationarity for optimal control of the
  obstacle problem with control constraints}, SIAM J.~Optim., 24 (2014),
  pp.~1914--1932.

\bibitem{Wachsmuth2016:2}
{\sc G.~Wachsmuth}, {\em A guided tour of polyhedric sets}, J. Convex Anal., 26
  (2019), pp.~153--188,
  \url{http://www.heldermann.de/JCA/JCA26/JCA261/jca26010.htm}.

\end{thebibliography}


\end{document}